\documentclass{compositio}
\usepackage{amsmath, amsfonts, amssymb, amscd, latexsym, graphicx, psfrag, color, float, tikz}
\usepackage[all]{xy}
\usepackage{ifpdf}

\newtheorem{dummy}{dummy}[section]
\newtheorem{lemma}[dummy]{Lemma}

\newtheorem{theorem}[dummy]{Theorem}

\newtheorem{proposition}[dummy]{Proposition}
\theoremstyle{definition}
\newtheorem{definition}[dummy]{Definition}
\newtheorem{example}[dummy]{Example}
\newtheorem{remark}[dummy]{Remark}

\newcommand{\bA}{\mathbb{A}}
\newcommand{\bC}{\mathbb{C}}

\newcommand{\bG}{\mathbb{G}}

\newcommand{\bQ}{\mathbb{Q}}
\newcommand{\bR}{\mathbb{R}}
\newcommand{\bZ}{\mathbb{Z}}

\newcommand{\cB}{\mathcal{B}}

\newcommand{\cE}{\mathcal{E}}
\newcommand{\cF}{\mathcal{F}}

\newcommand{\cI}{\mathcal{I}}

\newcommand{\cL}{\mathcal{L}}

\newcommand{\cQ}{\mathcal{Q}}
\newcommand{\cR}{\mathcal{R}}

\newcommand{\cS}{\mathcal{S}}
\newcommand{\cT}{\mathcal{T}}
\newcommand{\cU}{\mathcal{U}}

\newcommand{\Perf}{\mathcal{P}\mathrm{erf}}

\newcommand{\op}{\mathrm{op}}

\DeclareMathOperator{\Edge}{Edge}

\begin{document}

\title[Topological Fukaya category and mirror symmetry for punctured surfaces
]{Topological Fukaya category and mirror symmetry for punctured surfaces
}

\begin{abstract} In this paper we 
establish a version of homological mirror symmetry for punctured Riemann surfaces. Following a proposal 
of Kontsevich we model A-branes on a punctured surface $\Sigma$ via the  topological Fukaya category. We prove that the topological Fukaya category of $\Sigma$ is equivalent to the category of matrix factorizations of a certain mirror LG model $(X,W)$. Along the way we establish new gluing 
results for the topological Fukaya category of punctured surfaces which are of independent interest.   
\end{abstract}

\author{James Pascaleff}
\email{jpascale@illinois.edu}
\address{Department of 
Mathematics, University of Illinois at Urbana-Champaign, 1406 West Green Street, Urbana, IL 61801, US}

\author{Nicol\`o Sibilla}
\email{N.Sibilla@kent.ac.uk}
\address{Max Planck Institute for Mathematics, Bonn, Germany}
\curraddr{School of Mathematics, Statistics and Actuarial Science (SMSAS), University of Kent, Sibson Building, Parkwood Road, Canterbury, CT2 7FS, UK}

\maketitle


\section{Introduction}
The Fukaya category is an   intricate invariant of symplectic manifolds. 
One of the many subtleties of the theory is that pseudo-holomorphic discs, which control compositions of morphisms in the Fukaya category, are global in nature. As a consequence,  there is no way to calculate  the Fukaya category of a general symplectic manifold by breaking it down into local pieces.\footnote{See however recent proposals of Tamarkin \cite{Ta} and Tsygan \cite{Ts}.}  
In the case of exact symplectic manifolds, however, the Fukaya category is expected to have good local-to-global properties. For instance, if $S=T^*M$ is 
the cotangent bundle of an analytic variety, this follows from work of Nadler and Zaslow. They prove in \cite{NZ, N0}  
that the (infinitesimal) Fukaya category of $S$ is equivalent to the category of constructible sheaves on the base manifold $M$. This implies in particular that the Fukaya category of $S$ localizes as a sheaf of categories over $M$.

Recently Kontsevich \cite{Ko} has proposed that the Fukaya category of a Stein manifold $S$ can be described in terms of a (co)sheaf of categories on a skeleton of $S$. A skeleton is, roughly, a half-dimensional CW complex $X$ embedded in $S$ as a Lagrangian deformation retract. According to Kontsevich,  $X$ should carry a cosheaf of categories, which we will denote $\cF^{top}$, that  encodes in a universal way the local geometry of the singularities of $X$. He conjectures that the global sections of 
$\cF^{top}$ on $X$ should be equivalent to the  wrapped  Fukaya category of $S$.

Giving a rigorous definition of  the cosheaf 
$\cF^{top}$ is subtle. Work of several authors has clarified the case
of punctured Riemann surfaces \cite{DK, N, STZ}, while 
 generalizations to higher dimensions have been 
 pursued in \cite{N3, N4}. The theory is considerably easier in complex dimension one  because skeleta of punctured Riemann surfaces, also known as ribbon graphs or spines, have a simple and well studied combinatorics and geometry, while the higher dimensional picture is only beginning to emerge \cite{RSTZ, N5}. Implementing Kontsevich's ideas, the formalism developed in \cite{DK, N, STZ} defines a covariant functor $\cF^{top}(-)$ from a category of ribbon graphs and open inclusions to triangulated dg categories.

An important feature of the theory is that, if $X$ and $X'$ are two distinct compact skeleta of a punctured surface 
$\Sigma$, there is an equivalence 
$$
\cF^{top}(X) \simeq \cF^{top}(X'). 
$$
We will refer to $\cF^{top}(X)$ as the \emph{topological Fukaya category} of $\Sigma$, and we denote it $Fuk^{top}(\Sigma)$. In this paper we take $Fuk^{top}(\Sigma)$ as 
a model for the category of A-branes on $\Sigma$. We prove homological mirror symmetry for  punctured Riemann surfaces by showing that $Fuk^{top}(\Sigma)$ is equivalent to the category of B-branes on a mirror geometry LG model.  

\subsection{Hori-Vafa homological mirror symmetry} 
Let us review the setting of Hori-Vafa mirror symmetry for LG models \cite{HV, GKR}. Let $X$  be a toric threefold with trivial canonical bundle. The fan of $X$ can be realized as a smooth subdivision of the cone over a two-dimensional lattice polytope, see Section \ref{sec tcy} for more details. The height function on the fan of $X$ gives rise to a regular map 
$$
W: X \rightarrow \bA^1, 
$$
which is called the \emph{superpotential}. 
The category of B-branes for the LG model  $(X, W)$ is the $\bZ_2$-graded category of matrix factorizations $MF(X,W)$. The mirror of the LG-model 
$(X, W)$  is a smooth algebraic curve 
$\Sigma_W$ in $\bC^* \times \bC^*$,  called the \emph{mirror curve} (see Section \ref{Hori-Vafa}). The following is our main result.  
\begin{theorem}[Hori-Vafa homological mirror symmetry]
\label{mainthrm}
There is an equivalence 
$$
Fuk^{top}(\Sigma_W) \simeq MF(X, W). 
$$
\end{theorem}  
Theorem \ref{mainthrm} provides a proof of homological mirror symmetry for punctured surfaces, provided that we model the category of A-branes via the topological Fukaya category. This extends  to all genera earlier results for curves of genus zero and one which were obtained in \cite{STZ} and \cite{DK}. We also mention work of 
Nadler, who studies both directions of Hori-Vafa mirror symmetry for higher dimensional pairs of pants \cite{N2, N3, N4}.

We learnt the statement of Hori-Vafa homological 
mirror symmetry for punctured  surfaces from the inspiring paper \cite{AAEKO}. In \cite{AAEKO} the authors prove homological mirror symmetry for punctured spheres. Their main theorem is parallel to our own (in genus zero) with the important difference that they work with the wrapped Fukaya category, rather than with its topological model. See also related work of Bocklandt \cite{Bo}.  Mirror symmetry for higher-dimensional pairs of pants was studied by Sheridan in \cite{Sh}.

Denote $Fuk^{wr}(\Sigma)$ the wrapped Fukaya category of a punctured surface $\Sigma$. Our main result combined with the main result of \cite{AAEKO} gives equivalences
$$
Fuk^{top}(\Sigma_W) \simeq MF(X, W) \simeq Fuk^{wr}(\Sigma_W), 
$$ 
for all Riemann surfaces 
$\Sigma_W$ which can be realized as unramified cyclic covers of punctured spheres. Thus, for this class of examples, the topological Fukaya category captures the wrapped Fukaya category, corroborating Kontsevich's proposal. We also remark that a proof of the equivalence between topological and 
wrapped Fukaya category, with different methods, appeared in the recent \cite{HKK}.

\begin{remark}  
When we were close to completing the project we learnt that Lee, in her thesis \cite{Le}, extends  the results of \cite{AAEKO} to all genera. Although our techniques are very different, conceptually the  
approach pursued in this work and in  Lee's are  closely related. 
The results of this paper are logically independent of those of \cite{Le}, 
since we use the topological version of the Fukaya category instead of
the version defined in terms of pseudo-holomorphic curves. 
\end{remark}

\begin{remark}
\label{firstlastremark}
Theorem \ref{mainthrm} gives a homological mirror symmetry picture for the wrapped Fukaya category of punctured Riemann surfaces. Our techniques can also be used to obtain more general mirror symmetry statements for \emph{partially wrapped} Fukaya categories: this corresponds to considering non-compact skeleta having semi-infinite edges that approach the punctures of the Riemann surface. In the parlance of partially wrapped Fukaya categories, the non-compact ends of the skeleton are called \emph{stops}: they encode the directions along which wrapping is not allowed.

For brevity, we limit ourselves to a sketch of the theory in this more general setting.  
Let  $S \subset \Sigma_W$ be such a skeleton with non-compact edges. Then $S$ determines a (non unique) \emph{stacky} partial \emph{compactification} $\widetilde{X}$ of $X$, which is no longer Calabi--Yau. The punctures of 
$\Sigma_W$ are in bijection with the non-proper irreducible components of the singular locus of $W$. These are all isomorphic to $\mathbb{A}^1$: considering a skeleton $S$ with $n$ non-compact edges approaching a given puncture corresponds to compactifying that copy of $\mathbb{A}^1$ to a stacky rational curve 
$\mathbb{P}^1(1, n)$.  

The superpotential in general will not extend to $\widetilde{X}$, and thus we cannot work with categories of matrix factorizations. We can however formulate our theory 
in terms of the more flexible formalism of the \emph{derived category of singularities}. Recall that if 
$X_0 = W^{-1}(0) \subset X$ is the fiber at $0 \in \mathbb{A}^1$ of the superpotential, by Orlov's Theorem \cite{O1}, there is an equivalence 
$$
MF(X, f) \simeq D_{sg}(X_0),
$$
where $D_{sg}(X_{0})$ denotes the derived category of singularities of $X_{0}$. Although the category of matrix factorizations might not be well defined after compactifying to 
$\widetilde{X}$, we can still make sense of the   derived category of singularities. 
Namely, 
denote by $\widetilde{X}_0 \subset \widetilde{X}$ the compactification of the the zero fiber 
$X_0$. 
The claim is that there is an equivalence of categories 
\begin{equation}
\label{partwraphms}
\cF(S) \simeq D_{sg}(\widetilde{X}_0)
\end{equation}
We expect that we could prove the homological mirror symmetry equivalence (\ref{partwraphms}) exactly in the same way as we prove Theorem \ref{mainthrm}. We refer the reader to Remark \ref{lastremark} in the main text for additional details.   
Let us also mention recent work of Lekili--Polishchuk who give 
a different picture of homological mirror symmetry for partially wrapped Fukaya categories of punctured surfaces in terms of \emph{Auslander orders} \cite{LeP}. 
\end{remark}
\subsection{The topological Fukaya category and pants decompositions} 
The technical core of the paper is a study of the way in which the topological Fukaya category interacts with pants decompositions. By construction 
$\cF^{top}(-)$ is a cosheaf of categories on the spine of a punctured surface. So locality is built into the definition of the topological Fukaya category.  
From a geometric perspective, this locality corresponds  to cutting up the surface into flat polygons having their vertices at the punctures.

In this paper we  prove that the topological Fukaya category of a 
punctured surface  satisfies also a different kind of local-to-global behavior: it  can be glued together from the Fukaya categories of the pairs of pants making up a pants decomposition of it. We believe that this result is of independent interest. We expect this to be a  feature of the topological Fukaya category in all dimensions, and we will return to this in future work. Based on recent parallel advances in the theory of the wrapped Fukaya category 
\cite{Le}, this seems to be a promising avenue to compare the wrapped and the topological pictures of the category of A-branes on Stein manifolds. 
In order to explain the gluing formula for pants decompositions 
we need to sketch first a construction that attaches to a tropical curve $G$ a category $\cB(G)$, full details can be found in 
Section \ref{plagrama}.

Let $\kappa$ be the ground field. 
We denote by $MF(X,f)$ the category of matrix factorizations of the function $f: X \to \bA^1_\kappa$ and by $Fuk^{top}(\Sigma)$ the topological Fukaya category of $\Sigma$.
We denote by $MF^{\infty}(X,f)$ and $Fuk^{top}_\infty(\Sigma)$  the Ind completions  of these categories.
 We attach to a vertex $v$ of $G$ the category 
$$
\cB(v) :=MF^{\infty}(\bA^3_\kappa, xyz),
$$ and to an edge $e$ of $G$ the category 
$$
\cB(e):=MF^{\infty}(\bG_m \times \bA_\kappa^2, yz), 
$$ 
where $y$ and $z$ are coordinates on $\bA_\kappa^2$. 
If a vertex $v$ is incident to an edge $e$ there is a restriction functor $\cB(v) \rightarrow \cB(e)$. We define $\cB(G)$ as the (homotopy) limit  of these restriction functors.

\begin{theorem}
\label{tlsba}
Let $\Sigma$ be an algebraic curve in $\bC^* \times \bC^ *$ and let $G$ be its tropicalization.  Then there is an equivalence
\begin{equation}
\label{eq111}
Fuk^{top}_\infty(\Sigma) \simeq \cB(G).
\end{equation}
\end{theorem}

As proved in \cite{P}, $MF^\infty(X,f)$ is a sheaf of categories for the \'etale topology on $X$. This gives rise to an expression for $MF^\infty(X,f)$ which is exactly parallel to (\ref{eq111}). Our main theorem follows easily from here. 

\subsection{The topological Fukaya category and closed covers} 
The proof of Theorem \ref{tlsba} hinges on the key observation that the cosheaf $\cF^{top}(-)$  behaves like a sheaf  with respect to a certain type of \emph{closed covers}. This somewhat surprising property of $\cF^{top}(-)$ is very natural  from the viewpoint of  mirror symmetry, because it is mirror to Zariski descent of quasi-coherent sheaves and matrix factorizations.

We   formulate our gluing result in terms of the 
\emph{Ind-completion} $\cF^{top}_{\infty}(-)$ of $\cF^{top}(-)$. In order to simplify the exposition, in the statement below we do not specify all our assumptions on the closed cover. We refer the reader to the main text for the precise statement of our main gluing result, 
Theorem \ref{gluing-main}.

\begin{theorem}
\label{trrrr}
Let $X$ be a ribbon graph. 
\begin{itemize}
\item If $Z$ is a closed subgraph of $X$ there are restriction functors
$$
R: \cF^{top}(X) \rightarrow \cF^{top}(Z), \quad  
R_\infty: \cF^{top}_\infty(X) \rightarrow \cF_\infty^{top}(Z). 
$$
\item Let $Z_1$ and $Z_2$ be closed subgraphs of $X$, such that $Z_1 \cup Z_2=Z$. Assume that the underlying topological space of the intersection $Z_{12} = Z_1 \cap Z_2$ is a disjoint union of copies of $S^1.$ Then, under suitable assumptions on $Z_1$ and $Z_2,$ the  diagram 
$$
\xymatrix{
\cF_\infty^{top}(X) \ar[r] \ar[d] & \cF_\infty^{top}(Z_1) \ar[d]\\
\cF_\infty^{top}(Z_2) \ar[r] & \cF_\infty^{top}(Z_{12}) }
$$
is a homotopy fiber product of dg categories. 
\end{itemize}
\end{theorem}

Restrictions to closed subgraphs for the topological Fukaya category have also been been considered by Dyckerhoff in \cite{D}. From the perspective of the wrapped Fukaya category, they are closely related to the stop removal functors appearing in recent work of Sylvan \cite{Sy}. The technique of gluing the topological Fukaya category across a decomposition of skeleta 
into closed pieces also plays a central role in \cite{Ku}. Theorem \ref{trrrr} is a key ingredient in the proof of Theorem \ref{tlsba}, which also depends on a 
careful study of the geometry of skeleta under pants attachments. Indeed our proof of Theorem \ref{tlsba} hinges on a recursion where, at each step, we are simultaneously gluing in a  pair of pants and  deforming the skeleton on the surface to make it compatible with this gluing. The  topological analysis required for the argument is carried out in Section \ref{surface-topology}.

\subsection{The structure of the paper}
In section \ref{preliminaries} we fix notations on dg categories and ribbon graphs.  In Section \ref{Hori-Vafa}  we explain the set-up of Hori-Vafa mirror symmetry and prove a key decomposition of the category of matrix factorizations, which is a simple consequence of its sheaf properties.  Section \ref{sec top Fuk} contains a summary of the theory of the topological Fukaya category based mostly on \cite{DK}, while in Section \ref{restrictions closed} we study restrictions functors to open and closed subgraphs, and their compatibilities. In 
Section \ref{sec:gluing} we prove that the topological Fukaya category can be glued from a special kind of covers by closed sub-graphs. Section \ref{surface-topology} is devoted to a careful examination of the interactions between ribbon graphs and pants decompositions. This play a key role in the proof of the main theorem, which is contained in Section \ref{The induction}.

\begin{acknowledgements} We thank 
 David Ben-Zvi, Dario Beraldo, Ga\"etan Borot,  Tobias Dyckerhoff, Mikhail Kapranov, Gabriel Kerr,  Charles Rezk, Sarah Scherotzke, Paul Seidel, and Eric Zaslow for useful discussions and for  their interest in this project. Yank{\i} Lekili helped us with Remark \ref{firstlastremark} and \ref{lastremark}. We also thank the anonymous referee for comments that significantly improved the exposition. This project started when both authors were visiting the Max Planck Institute for Mathematics in Bonn in the Summer of 2014, and they thank the institute for its hospitality and support. JP was partially supported by NSF Grant DMS-1522670. NS thanks the University of Oxford and Wadham College, where part of this work was carried out, for excellent working conditions. 
\end{acknowledgements}

\section{Notations and conventions}
\label{preliminaries}
\begin{itemize}
\item We fix throughout a ground field $\kappa$ of characteristic $0$. 
\item Throughout the paper we  will work with small and large categories, and the categories that they form.  We will ignore  set-theoretic issues, and limit ourselves to mention that they can be formally obviated by placing oneself in appropriate Grothendieck universes, see for instance 
Section 2 of \cite{To}. \end{itemize}


\subsection{Categories} 
\label{catcategories}
In this paper we will work in the setting of \emph{dg categories}. The theory of ordinary triangulated categories suffer from limitations which make it unsuitable to 
many applications in modern 
geometry. These limitations are essentially of two kinds. First, there are issues arising  when working \emph{within} a given triangulated category, and  that originate from 
  the lack of functorial  universal constructions, such as cones.  Additionally, triangulated categories do not themselves form a well-behaved category. Among other things, 
  this implies that we have no workable notion of limits and colimits of triangulated categories, 
which are essential to perform \emph{gluing}  of categories.

The theory of dg categories gives us means to overcome 
both kinds of  limitations. As opposed to ordinary triangulated categories, dg categories fit inside homotopically enriched categories  and this allows us to perform operations such as taking limits and colimits. There are two main ways to define dg categories and to understand the structure of  the category that they form: \begin{enumerate} 
\item Dg categories can be defined as ordinary  categories  \emph{strictly enriched} over chain complexes. Dg categories then form an ordinary category equipped with a  \emph{model structure}. In fact, there are several meaningful  options for the model structure: the one that is most relevant for us is the \emph{Morita model structure}  considered in \cite{T} and \cite{To}. 
\item Dg categories can also be defined as $\kappa$-linear $\infty$-categories: that is, as 
$\infty$-categories that carry an action of  the symmetric monoidal  
$\infty$-category of chain complexes. 
Following this definition, dg categories form naturally an $\infty$-category. This approach is taken up for instance in the important recent reference \cite{GR}.
\end{enumerate}

Approaches (1) and (2) each have several distinct  advantages. It is often useful  to have recourse to both viewpoints simultaneously, and this is what we shall do in the present paper. Categories strictly enriched over chain complexes are much more concrete objects, and their homotopy limits in the model   category of dg categories can be calculated explicitly. On the other hand, perspective (2) gives rise to more natural definitions and constructions  and puts at our disposal the comprehensive foundations on 
$\infty$-categories provided by the work of Lurie \cite{Lu} and  Gaitsgory--Rozenblyum \cite{GR}.

In fact, there is a well-understood dictionary  between  these two viewpoints. Inverting weak equivalences in the model structure of dg categories (1) yields an $\infty$-category which is equivalent to 
the $\infty$-category of dg categories (2). As a consequence homotopy limits and colimits match, via this correspondence, their $\infty$-categorical counterparts. A proof of this equivalence was given by Cohn \cite{Co} and then in a more general context by Haugseng \cite{H} (see Section \ref{Cohn} below).  Throughout the paper we  adopt the flexible $\infty$-categorical  formalism coming from view-point (2)  while performing all actual calculations (e.g. of limits) in the model category of dg categories made available by  view-point (1).  We shall give more precise pointers to the literature below. 

Before proceeding, however, we remark that we will be  mostly interested in 
 \emph{$\mathbb{Z}_2$-graded}, rather than $\mathbb{Z}$-graded, dg categories.   
Some of the results which we will recall in the following  
have not been  stated explicitly in the literature in the $\mathbb{Z}_2$-graded setting. However, they can all be transported 
to the $\mathbb{Z}_2$-graded setting without 
variations. We 
 refer to Section 1 of \cite{DK} and Section 2 of \cite{D} for a detailed summary of the Morita homotopy theory of $\bZ_2$-graded dg categories, which is closely patterned on 
the  $\bZ$-graded theory developed in \cite{To}.   

The standard reference for $\infty$-categories is \cite{Lu}. We will follow closely the treatment of dg categories given in \cite{DG} and  \cite[I.1.10.3]{GR}, to which we refer the reader for additional details. In particular, the notations and basic facts below are all taken from these two references. The only    difference is that we will systematically place ourselves in the $\mathbb{Z}_2$-graded setting.

 We denote $\mathrm{Vect}^{(2)}$ the presentable, stable and symmetric monoidal $\infty$-category of $\mathbb{Z}_2$-periodic chain complexes of $\kappa$-vector spaces. We denote  $\mathrm{Vect}^{(2), \mathrm{fd}}$ its full subcategory of compact objects. For the rest of this section we will call:
\begin{itemize}
\item $\infty$-categories tensored over 
$\mathrm{Vect}^{(2), \mathrm{fd}},$ \emph{$\kappa$-linear $\infty$-categories} 
\item cocomplete $\infty$-categories tensored over 
$\mathrm{Vect}^{(2)},$ \emph{cocomplete $\kappa$-linear $\infty$-categories}. 
\end{itemize}
These are dg categories in the sense of view point (2). We introduce  the following notations:  
\begin{itemize} 
\item We denote $\mathrm{DGCat}^{(2), \mathrm{non-cocmpl}}$ the $\infty$-category of (not necessarily small) $\kappa$-linear  $\infty$-categories    
\item We denote $\mathrm{DGCat}^{(2)}_{\mathrm{cont}}$ the $\infty$-category of cocomplete $\kappa$-linear graded dg categories, and continuous functors between them.  $\mathrm{DGCat}^{(2)}_{\mathrm{cont}}$ carries a symmetric monoidal structure, and we denote 
its internal Hom by $\mathrm{Fun}_{\mathrm{cont}}(-,-).$ 
\item We denote  $\mathrm{DGCat}^{(2)}_{\mathrm{small}}$ the $\infty$-category of small 
$\kappa$-linear $\infty$-categories. $\mathrm{DGCat}^{(2)}_{\mathrm{small}}$ also carries a natural symmetric monoidal structure and we denote 
its internal Hom by $\mathrm{Fun}(-,-)$.
\item If $\mathrm{C}$ is an object in $\mathrm{DGCat}^{(2)}_{\mathrm{cont}},$ we denote $\mathrm{C}^\omega$ its subcategory of compact objects,  
$$
\mathrm{C}^\omega \in \mathrm{DGCat}^{(2)}_{\mathrm{small}}. 
$$ 
\end{itemize}
We have  the \emph{forgetful} and the \emph{Ind-completion} functors 
$$
\mathrm{U}: \mathrm{DGCat}^{(2), \mathrm{non-cocmpl}}  \longrightarrow  \mathrm{DGCat}^{(2)}_{\mathrm{cont}}  \quad   \quad \mathrm{Ind}: \mathrm{DGCat}^{(2), \mathrm{non-cocmpl}}  \longrightarrow \mathrm{DGCat}^{(2)}_{\mathrm{cont}}. 
$$
The Ind-completion functor can be defined, on objects, via the formula (see \cite[I.1]{GR},  Lemma 10.5.6)  
$$ 
\mathrm{C} \in \mathrm{DGCat}^{(2), \mathrm{non-cocmpl}}  \mapsto \mathrm{Ind}(\mathrm{C})= \mathrm{Fun}_\kappa(\mathrm{C}^{\mathrm{op}}, \mathrm{Vect}^{(2)})  \in \mathrm{DGCat}^{(2)}_{\mathrm{cont}}
$$ 
where $\mathrm{Fun}_k(-,-)$ denotes the $\infty$-category of $\mathrm{Vect}^{(2), \mathrm{fd}}$-linear functors. 

\begin{remark}
In the main text we will apply the Ind-completion functor only to \emph{small} $\kappa$-linear $\infty$-categories: that is, to objects in the full subcategory 
$$
\mathrm{DGCat}^{(2)}_{\mathrm{small}} \subset 
\mathrm{DGCat}^{(2), \mathrm{non-cocmpl}}. 
$$
\end{remark}

\begin{lemma}
\label{indsmallcol}
The functor $\mathrm{Ind}$ preserves small colimits of small categories. 
\end{lemma}
 \begin{proof}
 In the general setting of  $\infty$-categories, this is proved   as   \cite[I.7]{GR}, Corollary 7.2.7. The proof carries over to the $\kappa$-linear setting without variations. 
 \end{proof}

\subsubsection{Rigid models}
\label{Cohn}
As we have already mentioned, the two viewpoints (1) and (2) on  foundations of dg categories  which we discussed in Section \ref{catcategories} can be fully mapped onto each other. This will allow us to leverage the computational power of model structures, while retaining at the same time the flexibility of the $\infty$-categorical formalism. A precise formulation of the equivalence between perspectives (1) and (2) was given in work of Cohn \cite{Co}. We recall 
this below with the usual proviso that we will state it in the 
$\mathbb{Z}_2$-graded setting. 

We will sometimes refer to   categories  strictly enriched over chain complexes as \emph{(ordinary) dg categories} in order to differentiate them 
from $\kappa$-linear $\infty$-categories.  
%
Let $dgCat_\kappa^{(2)}$ be the model category of small categories strictly enriched over chain complexes equipped with the Morita model structure \cite{To}. Let $\mathcal{W}$ be the set of weak equivalences for the Morita model structure: these are dg functors which induce quasi-equivalences at the level of ($\mathbb{Z}_2$-graded) module categories. 
\begin{proposition}[\cite{Co}, Corollary 5.5]
\label{comparisondgenrich}
There is an equivalence of $\infty$-categories
$$
N(dgCat_\kappa^{(2)})[\mathcal{W}^{-1}] \simeq \mathrm{DGCat}^{(2)}_{\mathrm{small}}
$$
where $N(-)$ is the \emph{nerve}. 
\end{proposition}
 
We will also need an analogue of Cohn's result in the setting of \emph{large} categories, that is for $\mathrm{DGCat}^{(2)}_\mathrm{cont}.$ Let us denote 
$ 
dgCat^{(2)}_{\kappa, \mathrm{cont}}
$ 
the strictly enriched analogue  of $\mathrm{DGCat}^{(2)}_\mathrm{cont}.$ This is the Morita model category of cocomplete categories strictly enriched 
over $\kappa-mod^{(2)}.$ 
Then the analogue of Proposition \ref{comparisondgenrich} is discussed as  Example 5.11 of \cite{H}. In many situations having recourse to the model category of dg categories has great computational advantages. In particular,  we  will be able to compute $\infty$-categorical (co)limits in 
$\mathrm{DGCat}^{(2)}_\mathrm{cont}$ 
and $\mathrm{DGCat}^{(2)}_\mathrm{small}$ by calculating the respective homotopy (co)limits of the strict models in  
$dgCat^{(2)}_{\kappa, \mathrm{cont}}.$
 
Working with rigid models is often also helpful when considering commutative diagrams, as it allows us to sidestep issues of homotopy coherence in $\infty$-categories. Recall that a diagram of ordinary dg categories 
 $$
 \xymatrix{
 \mathrm{A} \ar[r]^-F \ar[d]_-G & \mathrm{B} \ar[d]^-H \\ 
 \mathrm{C} \ar[r]^-K   & \mathrm{D} }
 $$
 \emph{commutes}  if there is a natural transformation $\alpha: H \circ F \Rightarrow K \circ G,$ which becomes 
 a natural equivalence when passing to homotopy categories. All commutative diagrams of dg categories in this paper will be understood as coming from commutative diagrams of ordinary dg categories as above. 

%
%
%
%
%
 \subsubsection{Limits of dg categories}
Homotopy 
limits of dg categories can be calculated explicitly.  
Colimits of dg categories are, on the other hand, quite subtle. It  is therefore an important  observation that,   under suitable  assumptions,  colimits can actually be turned into limits. This key fact can be formulated in various settings. For instance, in the context of colimits of presentable $\infty$-categories such a result follows from Proposition 5.5.3.13 and 5.5.3.18 of \cite{Lu}.  We will  recall a formulation of this fact   in the setting of dg-categories   due to Drinfeld and Gaitsgory. We follow  closely Section 1.7.2 of \cite{DG}, with the usual difference that we adapt every statement to the $\mathbb{Z}_2$-graded setting. 

Let $I$ be a small $\infty$-category, and let $\Psi: I \to \mathrm{DGCat}_{\mathrm{cont}}^{(2)}$ be a functor. For every pair of objects $i, j \in I$ and 
 morphism $i \to j,$ we set 
$$
C_i:=\Psi(i) \quad C_j:=\Psi(j) \quad \psi_{i, j} := \Psi(i \to j): C_i \to C_j. 
$$
Assume that each functor $\psi_{i,j}$ admits a continuous right adjoint $\phi_{j,i}.$ Then there is a  functor $\Phi: I^{op} \to \mathrm{DGCat}_{\mathrm{cont}}^{(2)}$ such that
$$
\Phi(i)=C_i, \quad \text{and} \quad \Phi(i \to j)=\phi_{j,i}. 
$$
\begin{proposition}[Proposition 1.7.5 \cite{DG}]
\label{limit=colimit}
The colimit
$$
\varinjlim_{i \in I} C_i := \varinjlim_{I} \Psi \in \mathrm{DGCat}_{\mathrm{cont}}^{(2)}
$$
is canonically equivalent to the limit 
$$
\varprojlim_{i \in I^{op}} C_i := \varprojlim_{I} \Phi \in \mathrm{DGCat}_{\mathrm{cont}}^{(2)}. 
$$
\end{proposition}

%
%
%
%
%
%
\subsubsection{Sheaves of dg categories}
We will make frequent recourse to the concept of \emph{sheaf of dg categories} either on graphs, or on schemes equipped with the Zariski topology.  
For our purposes it will be enough to use  a rather weak notion of sheaf, which does not take into consideration hypercovers.  Let $\cF$ be a presheaf of dg categories over a topological space $X.$ We say 
that $\cF$  satisfies  \emph{\v{C}ech descent} if,  
for all open subset $U \subset X$  and cover $\cU=\{U_i\}_{i \in I}$ of $U,$ the restriction functor 
\begin{equation}
\label{cek}
\xymatrix{
\cF(U)  \ar[r]
& 
\Big( \cF (\cU) 
\ar@<-.5ex>[r]_-*!/d0.5pt/{\labelstyle  
}
\ar@<.5ex>[r]^ -*!/u0.7pt/{\labelstyle 
} 
& \cF  (\cU \times_U \cU) 
\ar@<-.5ex>[r]_-*!/d0.5pt/{\labelstyle  
}
\ar@<.0ex>[r]^ -*!/u0.7pt/{\labelstyle 
} 
\ar@<.5ex>[r]^ -*!/u0.7pt/{\labelstyle 
} & 
 \ldots \Big) }
\end{equation}
realizes $\cF(U)$ as the limit of the semi-cosimplicial object obtained by evaluating $\cF$ on the \v{C}ech nerve of $\cU$. If 
$\cF$ satisfies  \v{C}ech descent, we say that $\cF$ is a \emph{sheaf} of dg categories. The dual notion of \emph{cosheaf} of dg categories is defined by reversing all the arrows in  (\ref{cek}). 

There are two conditions under which (\ref{cek}) can be simplified, and that will occur in the examples we will be considering in the rest of the paper:
\begin{itemize}
\item[(a)] when triple and higher intersections of the cover $\cU$ are empty,
\item[(b)] when $\cF$ evaluated on the triple and higher intersection of the cover $\cU$ is equivalent to the zero category.   
\end{itemize}
We remark that condition (a) is in fact of special case of condition (b), but it is useful to keep them distinct. In this paper we will consider (co)sheaves of  categories on graphs, which have covering dimension one and where therefore condition (a) is automatically satisfied. 
 Also, we will consider sheaves of matrix factorizations on three-dimensional toric varieties equipped with the Zariski topology: there, although triple overlaps will not be empty,  condition (b) will apply.

 Let $\cU=\{U_1, \ldots, U_n\}$ be an open cover of $U$ and assume that either condition (a) or (b) are satisfied. Then the sheaf axiom  (\ref{cek})  is equivalent to requiring that  $\cF(U)$ is the limit of the following  much smaller diagram, which coincides with the truncation of the \v{Cech} nerve encoding the sheaf axiom in the classical setting of sheaves of sets 
\begin{equation}
\label{fincek}
\xymatrix{
\cF(U)  \ar[r]
& 
\Big( \displaystyle \prod_{i = 1, \ldots, n} \cF (U_i) 
\ar@<-.5ex>[r]_-*!/d0.5pt/{\labelstyle  
}
\ar@<.5ex>[r]^ -*!/u0.7pt/{\labelstyle 
} 
& \displaystyle \prod_{i < j} \cF  ( U_i \cap U_j)   \Big) 
 }
\end{equation}
The same reasoning applies to cosheaves of categories, with the only difference that the arrows in (\ref{fincek}) have to be reversed.

We will always work in settings where either (a) (in the case of the topological Fukaya category) or (b) (in the case of   matrix factorizations) are satisfied,  and therefore we always implicitly reduce to (\ref{fincek}).    

\subsubsection{Schemes and stacks}
If $X$ is a scheme or stack we denote 
$$
\Perf^{(2)}(X) \in \mathrm{DGCat}^{(2)}_{\mathrm{small}},  \quad 
\cQ Coh^{(2)}(X)  \in \mathrm{DGCat}^{(2)}_{\mathrm{cont}}
$$
the \emph{$\bZ_2$-periodization} of the dg categories of perfect complexes and of quasi-coherent sheaves on $X$. We refer the reader to   
Section 1.2 of \cite{DK} for the definition of  
 $\bZ_2$-periodization. All schemes appearing in this paper will be quasi-compact and with affine diagonal. All DM stacks will be global quotients of such schemes by affine groups.  Using the terminology introduced in 
 \cite{BFN}, these are all examples \emph{perfect stacks}. This implies in particular that, if $X$ is such a scheme or stack, the category of quasi-coherent sheaves over $X$ is equivalent to  the Ind-completion of its category of perfect complexes
 $$
 \cQ Coh^{(2)}(X) \simeq 
 \mathrm{Ind}(\Perf^{(2)}(X)). 
 $$
\subsection{Ribbon graphs}  
For a  survey of the theory ribbon graphs see \cite{MP}, and Section 3.3 of \cite{DK}. We will just  review some standard terminology.  A  graph $X$ is a pair $(V, H)$ of finite sets equipped with the following extra data:
\begin{itemize}
\item An involution $\sigma: H \rightarrow H$
\item A map $I: H \rightarrow V$
\end{itemize}
We call $V$ the set of  vertices, and $H$ the set of  half-edges.  
Let $v$ be a vertex. We say that the half-edges in $I^{-1}(v)$ are   incident  to $v$. The cardinality of $I^{-1}(v)$ is called the valency  of the vertex $v$. The edges  of $X$ are the equivalence classes of half-edges under the action  of 
$\sigma$. We denote $E$ the set of edges of $\sigma$. The set of external edges of $X$ is the subset $E^ o \subset E$ of equivalence classes of cardinality one, which correspond to the fixed points of $\sigma$. The internal edges of $X$ are the elements of $E - E^ o$. Subdividing an edge $e$ of $X$ means adding to $X$ a two-valent vertex lying on $e$. More formally, let $e$ be equal to $\{h_1, h_2\} \subset H$. We add a new vertex $v_e$ to $V$, and two new half-edges $h_1'$ and $h_2'$ to $H$. We modify the maps 
$\sigma$ and $I$ by setting 
$$
\sigma(h_1) = h_1', \quad \sigma(h_2) = h_2', \quad 
I(h_1') = I(h_2') = v_e.
$$

It is often useful to view a graph as a topological space. This is done by modeling the external and the internal edges of $G$, respectively,  
as semiclosed and closed 
intervals, and gluing them according to the incidence relations. We refer to this topological model as the underlying topological space of $X$. When talking about the embedding of a graph $X$ into a topological space, we always mean the embedding of its underlying topological space.  

\begin{definition}
A \emph{ribbon graph} is a graph $X = (V, H)$ together with the datum of a cyclic ordering of the set $I^{-1}(v)$, for all vertices $v$ of $X$. 
\end{definition} 

If a graph $X$ is embedded in an oriented surface it acquires 
 a canonical ribbon graph structure by ordering the edges at each vertex counter-clockwise with respect to the orientation. Conversely, it is possible to attach to any ribbon graph $X$ a non-compact oriented surface inside which $X$ is embedded as a strong deformation retract. See \cite{MP}  for additional details on these constructions. If  $\Sigma$ is a Riemann surface, 
 a \emph{skeleton} or \emph{spine} of $\Sigma$ is a ribbon graph 
 $X$ together with an embedding $X \rightarrow \Sigma$  as a strong deformation retract.

\section{Hori-Vafa mirror symmetry}
\label{Hori-Vafa}
In this section we review the setting of 
mirror symmetry for toric Calabi-Yau LG models in dimension three. Mirror symmetry for LG models was first proposed by Hori and Vafa \cite{HV}, and is the subject 
of a vast literature in string theory and mathematics, see \cite{GKR} and references therein. In this paper we compare the category of B-branes on toric Calabi-Yau LG models and  the category of A-branes on the  mirror. 

\subsection{B-branes} 
\subsubsection{Toric Calabi-Yau threefolds}
\label{sec tcy}
Let $\widetilde{N}$ be a $n-1$-dimensional 
lattice, and let $P$ a lattice polytope 
in 
$\widetilde{N}_\bR = \widetilde{N} \otimes \bR$. 
Set 
$N := \widetilde{N} \oplus \bZ$, and 
$N_\bR := N \otimes \bR$. Denote 
$C(P) \subset N_{\bR}$ the cone over the polytope $P$ placed at height one in $N_\bR$. More formally, consider 
$$
\{1\} \times P \subset N_{\bR} \cong \bR \oplus \widetilde{N}_\bR, 
$$
and let $C(P)$ be the cone generated by $\{1\} \times P$ inside 
$N_\bR$. Let $F(P)$ be the fan consisting of $C(P)$ and all its faces. The affine toric variety 
$X_P$ corresponding to  $F(P)$ has an isolated Gorenstein singularity.  The toric resolutions of $X_P$ are in bijection with smooth subdivisions of the cone $C(P)$. We will be interested in toric \emph{crepant} resolutions, that is,  resolutions with trivial canonical bundle. 

Toric crepant resolutions of $X_P$ 
are given by unimodular triangulations of $P$, i.e. triangulations of $P$ by elementary lattice simplices. Any such triangulation $\cT$ gives rise to a smooth subdivision of the cone $C(P)$. We denote $C(\cT)$ the set of cones on the simplices $T \in \cT$ placed at height one in $N_\bR$. Let  $F(\cT)$ be the corresponding fan, and let 
$X_{ \cT }$ be the toric variety with fan $F(\cT)$. The variety $X_{ \cT }$ is smooth and Calabi-Yau. All toric crepant resolutions of $X_{P}$ are isomorphic to $X_{ \cT }$ for some unimodular triangulation $\cT$ of $P$. 

The following definition  will be useful  later  on, see for instance 
\cite{BJMS} for additional details on this construction. 
 \begin{definition}
Assume that $n=3.$ We denote $G_\cT$ the tropical curve dual to the triangulation $\cT$ of $P$. 
\end{definition}

Let $\widetilde{M} = Hom(\widetilde{N}, \bZ)$ and $M = Hom(N, \bZ)$.  
The height function on $N$ is by definition the projection 
$$
N = \bZ \times \widetilde{N} \rightarrow \bZ. 
$$
The height function corresponds to the lattice point 
$
(1, 0) \in M \cong \bZ \times \widetilde{M},  
$
which determines a monomial function 
$$
W_\cT: X_{\cT} \rightarrow \bA_\kappa^1. 
$$ 
The category of B-branes on the Landau-Ginzburg model $(X_{ \cT }, W_\cT)$ is the category of matrix factorizations for $W_\cT$, $MF(W_\cT)$. We review the theory of matrix factorizations next.

\subsubsection{Matrix factorizations} 
Let $X$ be a scheme or a smooth DM-stack and let 
$f: X \rightarrow \bA^1_\kappa$ be a regular function. The category of matrix factorizations for the pair 
$(X,f)$ was defined in \cite{LP} and \cite{O2},  extending the theory of matrix factorizations for algebras that goes back to classical work of Eisenbud \cite{E}. These references make various assumptions on $f$ and $X$, which are always satisfied in the cases we are interested in. 
In the following  
$X$ will always be smooth of finite type, and 
$f$ will be flat. 
 We will work with a dg enhancement of the category of matrix factorizations, which has been studied for instance in \cite{LS} and \cite{P}. We refer to these papers for additional details. We denote $MF(X,f)$ the $\bZ_2$-periodic 
dg-category of matrix factorizations of the pair 
$(X,f)$. It will often be useful to work with Ind-completed categories of matrix factorization. 

\begin{definition}
We denote $MF^\infty(X, f)$ the Ind-completion of $MF(X,f)$,
$$
MF^\infty(X, f) = \mathrm{Ind}(MF(X,f)) \in 
\mathrm{DGCat}^{(2)}_{\mathrm{cont}}. 
$$
\end{definition}
The category $MF^{\infty}(X, f)$ has the following important descent property. 

\begin{proposition}[\cite{P} Proposition A.3.1]
\label{prop:descent}
Let $f: X \rightarrow \bA^1$ be a morphism. Then the assignment
$$
U \mapsto MF^\infty(U, f|_U)
$$
determines a sheaf for the \' etale topology. 
\end{proposition}

Using Proposition \ref{prop:descent} we can give  
give a very concrete description of $MF^{\infty}(X_\cT, W_\cT)$, where $X_\cT$ and $W_\cT$ are as in section \ref{sec tcy}. In order to do so, we need to explain how to attach a matrix factorizations-type category to a certain class  of planar graphs. This will require setting up some notations and preliminaries. 
 
Let $I$ be a set of cardinality three, say $I = \{a, b, c \}$.  
Denote 
$$
X_I  = Spec(\kappa[t_i, i \in I]) = Spec(\kappa[t_a, t_b, t_c]),
$$ and let $f$ be the regular function  
$$
f  = \times_{i \in I} t_i = t_a t_b t_c : X \longrightarrow  \bA^1_\kappa.   
$$
For all 
$j \in I$, let $I_j$ be the subset $I -\{ j \} \subset I$. 
Let $U_j$ be the open  subscheme $X - \{t_j = 0\}$, and let $\iota_j$ be the inclusion $U_j \subset X$. Denote $$
\iota_j ^*: MF^\infty(X_I, f) \longrightarrow  
MF^\infty(U_j, f|_{U_j})
$$
the restriction functor.  Let $f_j$  be the regular function 
$$
f_j  = \times_{i \in I_j} t_i: U_j \longrightarrow \bA^1_\kappa. 
$$
Note that $f|_{U_j}$ is given by $ t_j f_j$.  

\begin{proposition}
\label{prop:resres}
There are  equivalences of dg categories  
$$
MF^\infty (U_j, f|_{U_j}) \simeq MF^\infty(U_j, f_j)  \simeq \cQ Coh^{(2)}(\bG_m). 
$$
\end{proposition}
\begin{proof}
Recall that objects of 
$MF(U_j, f_j)$ are pairs of free finite rank vector bundles on $U_j$, and maps between them  
$$
\xymatrix{
\Big ( V_1 \ar@/^10pt/@{->}[r]^{d_1}  & V_2 
\ar@/^10pt/@{->}[l]_{d_2} \Big ),  
} 
$$ 
having the property that 
$$
d_1 \circ d_2 = f_j \cdot Id_{V_2}, \, \, \text{ and } 
d_2 \circ d_1 = f_j \cdot Id_{V_1} . 
$$
Thus the assignment 
$$
\xymatrix{
\Big ( V_1 \ar@/^10pt/@{->}[r]^{d_1}  & V_2 
\ar@/^10pt/@{->}[l]_{d_2} \Big )  
} \in MF(U_j, f_j) \mapsto \xymatrix{
\Big ( V_1 \ar@/^10pt/@{->}[r]^{t_j \cdot d_1}  & V_2 
\ar@/^10pt/@{->}[l]_{d_2} \Big )  
}   \in MF(U_j, f|_{U_j})
$$
determines an equivalence 
$$
MF (U_j, f|_{U_j}) \simeq MF (U_j, f_j).   
$$
The first equivalence is obtained by Ind-completion. 

The second equivalence follows from 
Kn\"orrer periodicity. For a very general formulation of Kn\"orrer periodicity see Theorem 9.1.7 (ii) of  \cite{P}.  Let us assume for convenience that 
$I = \{1, 2, 3\}$, and that $j = 1$. Then 
$U_j = \bG_m \times \bA^2_\kappa$, and 
$f_j = t_2 \cdot t_3$, where $t_2$ and $t_3$ are coordinate on the factor $\bA^2_\kappa$. 
By Kn\"orrer periodicity $MF(U_j, f_j)$ is equivalent to the $\bZ_2$-periodic category of perfect complexes on the first factor, $\bG_m$. This concludes the proof.  \end{proof}

\begin{remark}
\label{remres}
The equivalences constructed in Proposition \ref{prop:resres} are given by explicit functors, and do not rely on further choices. For the second equivalence,  this follows from the proof of Kn\"orrer periodicity in  \cite{P}. Abusing notation we sometimes denote  $\iota_j^*$ also the composition of the pull-back with the equivalences from Proposition \ref{prop:resres}. Thus we may write 
$$
\iota_j^*: MF^\infty(X_I, f) \rightarrow 
MF^\infty(U_j, f_j), \quad \iota_j^*: MF^\infty(X_I, f) \rightarrow 
\cQ Coh^{(2)}(\bG_m). 
$$
\end{remark}

We can abstract from Remark \ref{remres} a 
formalism of restriction functors which will be useful in the next section. If   
$L$ is a set of cardinality two, denote $$
X_L  = Spec(\kappa[t_l, l \in L][u, u^{-1}]) \cong \bG_m \times \bA^2_\kappa, 
$$ 
and let $f$ be the morphism 
$$
f  = \times_{l \in L} t_l : X_L \longrightarrow  \bA^1_\kappa.   
$$ 
\begin{definition}
\label{defresres}
Let $I$ and $L$ be sets of cardinality three and two respectively, and assume that we are given an embedding 
$L \subset I$. We denote 
$R^{MF}_\infty$  the composite: 
$$
\xymatrix{
R^{MF}_\infty: MF^\infty(X_I, f) \ar@/_15pt/@{->}[rr] \ar[r]^-{\iota^*_j} & MF^\infty(U_j, f_j) \ar[r]^-{\simeq}  & 
MF^\infty(X_L, f),    
} 
$$
where 
\begin{enumerate}
\item $\{j\} = I -L$, and $\iota_j^*$ is defined as in Remark \ref{remres}. 
\item The equivalence $MF^\infty(U_j, f_j) \simeq MF^\infty(X_L, f)$ is determined by the isomorphism of $\kappa$-algebras 
$$
\kappa[t_i, i \in I][t_j^{-1}] = \kappa[t_l, l \in L][t_j, t_j^{-1}] \stackrel{\cong}{\longrightarrow} \kappa[t_l, l \in L][u, u^{-1}] 
$$
that sends $t_l$ to $t_l$, and $t_j$ to $t_u$. 
\end{enumerate}
\end{definition}
\subsubsection{Planar graphs and matrix factorizations}
\label{plagrama}
Let $G$ be a trivalent, planar graph. Assume for simplicity that 
$G$ does not contain any loop. We will explain how to attach to $G$ a matrix factorization-type category. We denote $V_G$ the set of vertices of 
$G$, and  $E_G$  the set of edges. 
\begin{itemize}
\item Let $v \in V_G$, and take a sufficiently  small ball $B_v$ in $\bR^2$ centered at $v$. 
Then the set of connected components of $B_v - G$ has cardinality three, and we denote it $I_v$. 
\item Let $e \in V_G$, and take a sufficiently small ball $B_e$ centered at any point in the relative interior of $e$. The set of connected components of $B_v - G$ has cardinality two, and we denote it $L_e$. 
\end{itemize}

\begin{remark}
\label{remreminc}
Note that the sets $I_v$ and $L_e$ do not depend (up to canonical identifications) 
on $B_e$ and $B_v$. Further, if a vertex $v$ is incident to an edge $e$, there is a canonical embedding: 
$
L_e  \subset I_{v}. 
$
\end{remark}
We attach to each vertex and edge of 
$G$ a category of matrix factorizations in the following way:
\begin{itemize}
\item We assign to $v \in V_G$ the category 
$$
\cB(v) := MF^\infty(X_{I_v}, f)
$$
\item We assign to $e \in E_G$ the category 
$$
\cB(e) := MF^\infty(X_{L_e}, f)
$$
\end{itemize}

By Remark \ref{remreminc}, and Definition \ref{defresres}, if a vertex $v$ is incident to an edge $e$ we have a restriction functor
\begin{equation}
\label{eqeq123}
R^{ \cB } : \cB(v) \longrightarrow 
\cB(e).
\end{equation}
If two vertices $v_1$ and $v_2$ are incident to an edge $e$, we obtain a 
diagram of restriction functors
$$
\xymatrix{
\cB(v_1) \times  \cB(v_2)
\ar@<-.5ex>[r]_-*!/d0.5pt/{\labelstyle }
\ar@<.5ex>[r]^ -*!/u0.7pt/{\labelstyle } 
& \cB(e). 
}
$$
Running over the vertices and edges of $G$, we obtain a \v{C}ech-type diagram in 
$\mathrm{DGCat}^{(2)}_{\mathrm{cont}}$
\begin{equation}
\label{eqeq}
\xymatrix{
 \prod_{v \in V_G} 
\cB(v) 
\ar@<-.5ex>[r]_-*!/d0.5pt/{\labelstyle }
\ar@<.5ex>[r]^ -*!/u0.7pt/{\labelstyle } 
& \prod_{e \in E_G} \cB(e). 
}
\end{equation}

\begin{definition}
\label{defff}
We denote $\cB(G)$ the equalizer of diagram 
(\ref{eqeq}) in $\mathrm{DGCat}^{(2)}_{\mathrm{cont}}$. \end{definition}

We say that a subset $T \subset G$ is a subgraph of $G$ if 
\begin{itemize}
\item $T$ is a trivalent  graph
\item If $e$ is an edge of $G$ such that $T \cap e$ is 
non-empty, then $e$ is contained in $T$ 
\end{itemize}
Note that if $T$ is a subgraph of $G$ we have inclusions 
$V_T \subset V_G$ and $E_T \subset E_G$. We can define a restriction functor 
$$
\cB(G) \longrightarrow 
\cB(T).  
$$
This is obtained by considering the natural map between 
\v{C}ech-type diagrams given by the obvious projections 
$$
\xymatrix{
 \prod_{v \in V_G} 
\cB(v) 
\ar@<-.5ex>[r]_-*!/d0.5pt/{\labelstyle }
\ar@<.5ex>[r]^ -*!/u0.7pt/{\labelstyle } 
\ar[d] 
& \prod_{e \in E_G} \cB(e) \ar[d] \\ 
 \prod_{v \in V_T} 
\cB(v) 
\ar@<-.5ex>[r]_-*!/d0.5pt/{\labelstyle }
\ar@<.5ex>[r]^ -*!/u0.7pt/{\labelstyle } 
& \prod_{e \in E_T} \cB(e)
}
$$
It is useful to extend to a general pair of graphs $T \subset G$ the notation for restriction functors that we introduced in  (\ref{eqeq123}) in the case of a vertex and a neighboring edge:  
\begin{definition}
If $T$ is a subgraph of $G$, we denote the restriction functor 
$$ 
R^{ \cB } : \cB(G) \longrightarrow 
\cB(T).
$$
\end{definition}

The definition of $\cB(G)$ allows us to encode the category of matrix factorizations of a toric Calabi-Yau LG model in a simple combinatorial package. We use the notations of section \ref{sec tcy}. 

\begin{theorem}
\label{gluing MF}
Let $P$ be a planar lattice polytope, equipped with a unimodular triangulation $\cT$. Let $G_\cT$ be the dual graph of $\cT$. Then there is an equivalence of dg categories 
$$
MF^\infty(X_\cT, W_\cT) \simeq \cB(G_\cT). 
$$
\end{theorem}
\begin{proof}
Let $C$ be the set of maximal cones 
in the fan of $X_\cT$. 
Consider the standard open cover of $X_\cT$ by toric affine patches:   
$
\{U_{\sigma}\}_{\sigma \in C},$ $U_\sigma \cong \bA^3_\kappa$. By Proposition \ref{prop:descent} the category $MF^{\infty}(X_\cT, f_\cT)$ can be expressed as the limit of the \v{C}ech diagram for the open cover 
$\{U_{\sigma}\}_{\sigma \in C}$: the vertices  of this diagram are products of the categories 
$$
MF^{\infty}(U_{\sigma}, f_\cT|_{U_{\sigma}}), \quad \text{ and } \quad MF^{\infty}(U_{\sigma} \cap U_{\sigma'}, f_\cT|_{U_{\sigma} \cap U_{\sigma'}}),  
\quad \text{ $\sigma$, $\sigma' \in C$},  
$$ 
and the arrows are products of pullback functors.

Note that there is a natural bijection $\phi$ between 
the set $V_\cT$ of vertices of $G_\cT$ and $C$. Moreover the definition of 
$I_v$  gives an identification 
$
X_{I_v} \cong U_{\phi(\sigma)},  
$
and thus a canonical equivalence 
$$
\cB(v) \simeq MF^{\infty}(U_{\phi(v)}, f_\cT|_{\phi(v)}). 
$$
Similarly, by Remark \ref{remres}, if $v$ and $v'$ are two vertices of $G_\cT$ and $e$ is the edge connecting them, we obtain a commutative diagram
$$
\xymatrix{
\cB(v) \ar[rr]^-{\simeq} \ar[d]_-{R^{ \cB } } && MF^{\infty}(U_{\phi(v)}, f_\cT|_{\phi(v)}) \ar[d]^-{\iota^*} \\ 
\cB(e) \ar[rr]^-{\simeq} && MF^{\infty}(U_{\phi(v)} \cap U_{\phi(v')}, f_\cT|_{U_{\phi(v)} \cap U_{\phi(v')}})
}
$$
where the horizontal arrows are canonical equivalences, and $\iota^*$ is the restriction of matrix factorizations along the embedding
$$
\iota: U_{\phi(v)} \cap U_{\phi(v')} \rightarrow U_{\phi(v)}. 
$$
Thus the diagrams computing $\cB(G_\cT)$ and $MF^{\infty}(X_\cT, f_\cT)$ are equivalent,  and this concludes the proof. \end{proof}

\subsection{A-branes} 
\subsubsection*{}
As explained in \cite{HV}, the mirror of a toric Calabi-Yau LG model $(X_\cT, W_\cT)$ is a punctured Riemann surface $\Sigma_\cT$ embedded as an algebraic curve in $\bC^* \times \bC^*$. The graph $G_\cT$ is the tropicalization of $\Sigma_\cT$. Since we are interested in studying the A-model on $\Sigma_\cT$ we can disregard its complex structure and focus on its topology, which is captured by the genus and the number of punctures (see for instance \cite{BS} for an explicit algebraic  equation of $\Sigma_\cT$). These can be read off from $G_\cT$.  The genus of $\Sigma_\cT$ is equal to the number of relatively compact connected components in $\bR^2 - G_\cT$. The number of punctures of $\Sigma_\cT$ is equal to the number of unbounded edges of $G_\cT$. 

\begin{figure}[h]
  \centering
  \includegraphics[width=0.7\textwidth]{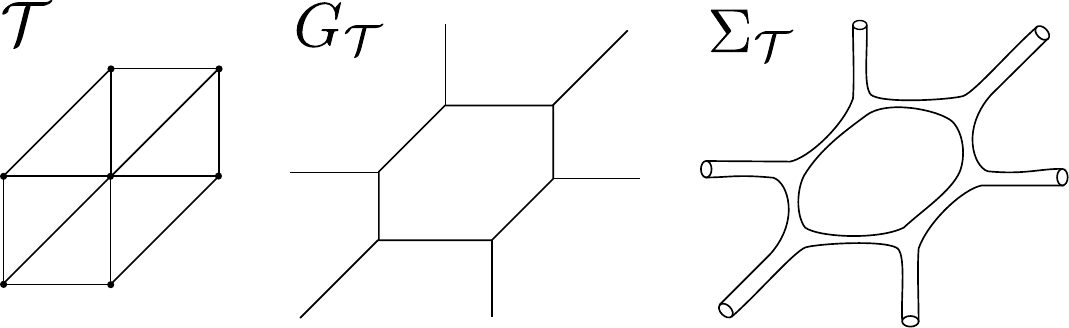}
  \caption{The picture shows an example of the relationship between the triangulation $\cT$, the tropical curve $G_\cT$, and the mirror curve $\Sigma_\cT$.}
  \label{fig:tropical}
\end{figure}

When $\Sigma_\cT$ has genus $0$, the authors of 
\cite{AAEKO} consider the wrapped Fukaya category of $\Sigma_\cT$, and prove that it is equivalent to the category of matrix factorizations of the mirror toric LG model. In this paper we will consider an alternative model for the category of $A$-branes on a 
punctured Riemann surface,  called the \emph{topological Fukaya category}. The construction of the  topological Fukaya category was first suggested by Kontsevich \cite{Ko} and was studied in   \cite{DK, N, STZ} and elsewhere in the case of punctured Riemann surfaces. 
We summarize the theory of the topological Fukaya category in Section \ref{sec top Fuk} below.

\section{The topological Fukaya  category}
\label{sec top Fuk}
In this Section we recall the definition of the topological Fukaya for punctured Riemann surfaces. We will mostly follow the approach of \cite{DK}, see also  \cite{N, STZ, HKK} for related alternative formulations of this theory. It will be important to consider the Ind-completion of the topological Fukaya category. We discuss this in Section \ref{sec:ind}. 

\subsection{The cyclic category and the topological Fukaya category}
\label{fuksmallcycsec}
We briefly review the setting of \cite{DK}. We refer for more details to the original paper and to \cite{D}. 

\begin{definition}[\cite{C}] Let $\Lambda$ be the category defined as follows: 
\begin{itemize}
\item The set of objects of $\Lambda$ is in bijection with the set of natural numbers. For all $n \in N$, $\langle n \rangle \in \Lambda$ is a copy of $S^1$ with $n+1$ marked points given by the $(n+1)$-th roots of  unity. 
\item A morphism 
$
\langle m \rangle \rightarrow \langle n \rangle 
$
is represented by a monotone degree one continuous map $S^1 \rightarrow S^1$ taking the $m+1$ marked points in the domain to the $n+1$ marked points in the range; such maps are considered up to homotopy relative to the marked points.
\end{itemize}
\end{definition}

\begin{proposition}
\label{prop:inter}
There is a natural equivalence of categories $(-)^*: \Lambda \rightarrow  \Lambda^{op}$. 
\end{proposition}
\label{prop:dual}
\begin{proof}
This equivalence is called \emph{interstice duality}. See  \cite{DK} Section 2.5. 
\end{proof}

\begin{definition}
\label{defcycocyc}
Let $C$ be an $\infty$-category. 
\begin{itemize}
\item 
A \emph{cyclic object} in $C$ is a  functor $ N(\Lambda^{op})  \rightarrow C$. We denote $C_\Lambda$ the $\infty$-category of cyclic objects in $C$. If $C=\cS$ is the $\infty$-category of spaces we denote  
$$L: N(\Lambda) \to \cS_\Lambda$$ the Yoneda functor. 
\item A \emph{cocyclic object} in $C$ is a functor $N(\Lambda)  \rightarrow C$. We denote $C^\Lambda$ the $\infty$-category of cocyclic objects in $C$. 
\end{itemize}
\end{definition}

Consider the map 
$
\bA^1_\kappa \stackrel{z^n}{\longrightarrow} \bA^1_\kappa$. 
We denote $\cE^{n}$ the ($\bZ_{n+1}$-graded category)  of matrix factorizations $MF(\bA^1_\kappa, z^{n+1}) \in \mathrm{DGCat}^{(2)}_{\mathrm{small}}$, see Section 2.3.5 of \cite{DK} for additional details.  
The categories $\cE^n$ assemble to a cyclic object in $\mathrm{DGCat}^{(2)}_{\mathrm{small}}$. We state the precise result below. 

\begin{proposition}[\cite{DK} Proposition 2.4.1, \cite{D} Theorem 3.2]
\label{prop E}
There is a cocyclic object 
$$
\cE^\bullet: N(\Lambda) \rightarrow \mathrm{DGCat}^{(2)}_{\mathrm{small}}
$$ that is defined on objects by the assignment 
$$
\langle n \rangle \in \Lambda \mapsto \cE^{n} \in \mathrm{DGCat}^{(2)}_{\mathrm{small}}.
$$ 
\end{proposition}

\begin{lemma}
\label{satadjoints}
The structure maps of the cocyclic object 
$\cE^\bullet$ admit a  left and a right adjoint. 
\end{lemma}
\begin{proof}
For all $n \in \mathbb{N}$ the category $\mathcal{E}^n$ is the $\mathbb{Z}_2$-periodization of a \emph{smooth} and \emph{proper} dg category, namely the dg category of representations of the $A_{n-1}$-quiver (see \cite{DK}, Theorem 2.3.6).  Now,   functors between smooth and proper dg categories always admit both a right and a left adjoint. In the  $\mathbb{Z}$-graded setting  a convenient recent reference for this fact is Theorem 7.4 \cite{Ge}, and the proof carries over   to the $\mathbb{Z}_2$-graded setting.  
\end{proof}

Let 
$$
(-)^{op  }: \mathrm{DGCat}^{(2)}_{\mathrm{small}} \rightarrow \mathrm{DGCat}^{(2)}_{\mathrm{small}}
$$ 
be the auto-equivalence of $\mathrm{DGCat}^{(2)}_{\mathrm{small}}$  sending a category to its opposite category. 
\begin{definition}
\label{E}
Denote $\cE_\bullet: \Lambda^{op} \rightarrow \mathrm{DGCat}^{(2)}_{\mathrm{small}}$ the cyclic object defined as the composite
$$
N(\Lambda^{op}) \stackrel{N(-)^*}{\longrightarrow} N(\Lambda) \stackrel{\cE^\bullet}{\longrightarrow} \mathrm{DGCat}^{(2)}_{\mathrm{small}} \stackrel{(-)^{op}}{\longrightarrow} \mathrm{DGCat}^{(2)}_{\mathrm{small}}.
$$ 
\end{definition}

\begin{remark}
\label{rem:morita}
There are several equivalent ways to define $\cE_\bullet$. We list them below.
\begin{enumerate}
\item Denote $(-)^\vee: (\mathrm{DGCat}^{(2)}_{\mathrm{small}})^{op} \rightarrow \mathrm{DGCat}^{(2)}_{\mathrm{small}}$ the functor mapping a category  $D$ in $\mathrm{DGCat}^{(2)}_{\mathrm{small}}$ to 
$\mathrm{Fun}(D, \mathrm{Vect}^{(2), \mathrm{fd}})$. Then 
$\cE_\bullet$ is equivalent to the cyclic object given by the composition: 
$$
\xymatrix{
N(\Lambda^{op}) \ar[r]^-{(\cE^{\bullet})^{op}} &  (\mathrm{DGCat}^{(2)}_{\mathrm{small}})^{op} \ar[r]^-{(-)^\vee} & \mathrm{DGCat}^{(2)}_{\mathrm{small}}.
}
$$
See \cite{DK} Section 3.2 for a discussion of this fact.  
\item By Lemma \ref{satadjoints} all the structure maps of $\cE^\bullet$ admit right adjoints. Passing to right adjoints yields a contravariant functor  
out of 
$N(\Lambda),$ which we denote 
$(\cE^{\bullet})^{R},$ 
$$
\xymatrix{
N(\Lambda^{op}) \ar[r]^-{(\cE^{\bullet})^{R}} &
\mathrm{DGCat}^{(2)}_{\mathrm{small}} 
}
$$
and which is also equivalent to the cyclic object  
$\cE_\bullet.$  
\end{enumerate}
\end{remark}

To any ribbon graph $X$, we may associate a cyclic set $\cL(X) \in \cS et_{\Lambda}$ (called the \emph{cyclic membrane} in \cite{DK}) as follows: Given a vertex $v$ of $X$, denote by $E(v)$ the set of edges at $v$, and set $B(v) = E(v)^{*}$. Let $\Lambda^{B(v)}$ be the cylic simplex corresponding to the cyclically ordered set $B(v)$. For an edge $e$ of $X$, we have a two element set $B(e) = V(e)^{*}$, where $V(e)$ is the set of vertices that $e$ joins. Let $\Lambda^{B(e)}$ be corresponding cyclic 1-simplex. Next, define the \emph{incidence category} $X_{[0,1]}$ of $X$: the set of objects of $X_{[0,1]}$ is the disjoint union of the set of vertices and the set of edges of $X$, and there is a unique morphism $v \to e$ for each flag $(v,e)$ in $X$ (consisting of a pair of a vertex $v$ and an edge $e$ incident to $v$). The cyclic simplices $\Lambda^{B(v)}$ and $\Lambda^{B(e)}$ assemble into a functor $U_{X} : X_{[0,1]}^{\op} \to \cS et_{\Lambda}$, and we define
\begin{equation}
  \cL(X) = \varinjlim U_X \in \cS et_{\Lambda}
\end{equation}

\begin{proposition}[Proposition 3.4.4 \cite{DK}]
  \label{prop:ribbon}
  The cyclic membrane construction
  $$
  \cL : \cR ib \rightarrow \cS et_{\Lambda}, \quad X  \in \cR ib \mapsto \cL(X).
  $$
  extends to a functor from the category $\cR ib$ of ribbon graphs and contractions between them to the category of cyclic sets.
\end{proposition}
 
If $X$ and $X'$ are ribbon graphs, a contraction 
$
X \rightarrow X'
$ 
is a map between the underlying topological spaces having the property that the preimage of each point in $X'$ is either a point, or a sub-tree of $X$. We refer the reader to \cite{DK} for additional details on the definition of $\cR ib$ and a proof of the proposition.

As explained in \cite{D} the functor constructed in Proposition \ref{prop:ribbon} can be enhanced to a functor of $\infty$-categories 
$$
\cL: N(\cR ib) \rightarrow \cS_\Lambda 
$$
where $\cS$ is the $\infty$-category of spaces. 


\begin{definition}
\label{kan}
\begin{itemize}
\item  Denote $\cF^ {\cE}: \cS_\Lambda \rightarrow  \mathrm{DGCat}^{(2)}_{\mathrm{small}}$ the $\infty$-categorical left Kan extension of $$
N(\Lambda) \stackrel{\cE^\bullet}{\longrightarrow} \mathrm{DGCat}^{(2)}_{\mathrm{small}} 
$$
along
$
N(\Lambda) \xrightarrow{L} \cS_\Lambda.
$
\item 
Denote 
$\cF_{\cE}: (\cS_\Lambda)^{op} \rightarrow \mathrm{DGCat}^{(2)}_{\mathrm{small}}$ the $\infty$-categorical right Kan extension of  $$
N(\Lambda^{op}) \stackrel{\cE_ \bullet}{\longrightarrow}\mathrm{DGCat}^{(2)}_{\mathrm{small}} 
$$
along $
N(\Lambda)^{op} \xrightarrow{(L)^{op}} (\cS_\Lambda)^{op}.
$
\end{itemize}
\end{definition}

Let $\Sigma$ be a Riemann surface with boundary and let $X \subset \Sigma$ be a spanning ribbon graph.   
The implementation of Kontsevich's ideas due to Dyckerhoff and Kapranov \cite{DK} (see also  \cite{N} and \cite{STZ}) gives ways to compute a model for the Fukaya category of $\Sigma$ from the combinatorics of $X$. We will refer to it as the topological (compact) Fukaya category of $X$ or sometimes as the topological (compact) Fukaya category of the pair $(\Sigma, X)$. The next definition gives the construction of the topological (compact) Fukaya category, see Definition 4.1.1 of \cite{DK}.

\begin{definition}
\label{top fuk}
Let $(\Sigma, X)$ be a punctured Riemann surface. 
\begin{itemize}

\item The \emph{topological  Fukaya category of $X$} is given by 
$$
\cF^ {top}(X) := \cF^{\cE}(\cL(X)).
$$

\item The \emph{topological compact Fukaya category} of $X$ is given by 
$$
\cF_{top}(X) := \cF_{\cE}(\cL(X))
$$

\end{itemize}
\end{definition}

\begin{remark}
The terminology we use in this paper differs slightly from the one introduced  in \cite{DK}. Namely 
\begin{itemize}
\item $\cF^{top}(X)$ in \cite{DK}  is called  the topological \emph{coFukaya category},
\item $\cF_{top}(X)$ in  \cite{DK} is called the topological \emph{Fukaya category}.
\end{itemize}
\end{remark}
\begin{proposition}
\label{proppropdual}
There is a natural equivalence
$$
\cF_ {top}(X) \simeq \mathrm{Fun}(\cF^ {top}(X), \mathrm{Vect}^{(2), \mathrm{fd}}).
$$
\end{proposition}
\begin{proof}
See the discussion after Definition 4.1.1 of \cite{DK}. 
\end{proof}


\subsection{The Ind-completion of the topological Fukaya category}
\label{sec:ind}
In this Section we introduce the Ind-completed version of the topological (compact) Fukaya category. This plays an important role in proving that $\cF^{top}(-)$ exhibits an interesting sheaf-like behavior with respect to closed coverings of ribbon graphs that we study in Section \ref{sec:closed}. We remark that Definitions \ref{ind E} and \ref{ind kan} mirror exactly Proposition \ref{prop E} and Definition \ref{E} and \ref{kan} from the previous Section, the only difference being that we are now working with cocomplete dg categories.

\begin{definition}
\label{ind E}
\begin{itemize}
\item 
Denote $\cI \cE^\bullet: N(\Lambda) \rightarrow \mathrm{DGCat}^{(2)}_{\mathrm{cont}}$ the cocylic object defined by the composition: 
$$
N(\Lambda) \stackrel{\cE^\bullet}{\longrightarrow} \mathrm{DGCat}^{(2)}_{\mathrm{small}} \stackrel{\mathrm{Ind}}{\longrightarrow} \mathrm{DGCat}^{(2)}_{\mathrm{cont}}.
$$ 
\item Denote $\cI \cE _\bullet: N(\Lambda^{op}) \rightarrow \mathrm{DGCat}^{(2)}_{\mathrm{cont}}$ the cyclic object defined by the composition: 
$$
N(\Lambda^{op}) \stackrel{\cE_\bullet}{\longrightarrow} \mathrm{DGCat}^{(2)}_{\mathrm{small}} \stackrel{\mathrm{Ind}}{\longrightarrow} \mathrm{DGCat}^{(2)}_{\mathrm{cont}}.
$$ 
\end{itemize}
\end{definition}       


\begin{remark}
\label{IndMoritaInd}
The same consideration of Remark \ref{rem:morita} apply to the Ind-completed setting. In particular, the structure maps of $\cI \cE^\bullet$ admit right and left adjoints, which are the Ind-completions of the right and left adjoints given by Lemma \ref{satadjoints}. Then  the cyclic object $\cI \cE_\bullet$ is equivalent to $(\cI \cE^\bullet)^R,$ which is defined exactly in the same way as in Remark \ref{rem:morita}. 
\end{remark}

\begin{definition}
\label{ind kan}
\begin{itemize}
\item  Denote $\cI \cF^ \cE: \cS_\Lambda  \rightarrow  \mathrm{DGCat}^{(2)}_{\mathrm{cont}}$ the $\infty$-categorical left Kan extension of 
$$
N(\Lambda) \stackrel{\cI \cE^\bullet }{\longrightarrow}\mathrm{DGCat}^{(2)}_{\mathrm{cont}}   
$$
along $N(\Lambda) \xrightarrow{L} \cS_\Lambda$. 
\item 
Denote 
$\cI \cF_{\cE}: (\cS_\Lambda)^{op} \rightarrow \mathrm{DGCat}^{(2)}_{\mathrm{cont}}$ the $\infty$-categorical right Kan extension of 
$$ 
N(\Lambda^{op}) \stackrel{\cI \cE_\bullet}{\longrightarrow} \mathrm{DGCat}^{(2)}_{\mathrm{cont}}.  
$$
along $N(\Lambda^{op}) \xrightarrow{(L)^{op}} (\cS_\Lambda)^{op}$. \end{itemize}
\end{definition}


\begin{proposition}
\label{prop:_top}
For all $X \in \cS_\Lambda$ there is an equivalence 
 $$
 \cI \cF ^\cE(X) \simeq \cI \cF _\cE(X). 
 $$ 
 \end{proposition}
\begin{proof}
This is a formal consequence of Proposition \ref{limit=colimit}. 
By Remark \ref{IndMoritaInd}, $\cI \cE_\bullet$ can be realized by taking right adjoints of the structure morphisms of $\cI \cE^\bullet$. If  $\cL$ is in $\cS_\Lambda,$  by the general formula for pointwise Kan extensions we have 
$$
 \cI \cF^\cE (\cL) 
\, \, = \varinjlim_{\{L(\langle n \rangle) \rightarrow \cL\}} \cI \cE^ n 
\, \, \simeq \varprojlim_{\{L(\langle n \rangle) \rightarrow \cL\}}  \cI \cE_ n  \, \, = \, \,  \cI \cF_\cE (\cL) \quad \text{in} \quad \mathrm{DGCat}^{(2)}_{\mathrm{cont}}. 
$$

\end{proof}

\begin{definition}
 \label{defprop:_top}
 If $X$ is a ribbon graph, we set 
$$
\cF_\infty^{top}(X):= \cI \cF ^\cE(\cL (X)) \simeq \cI \cF _\cE(\cL (X)). 
$$ 
We call $\cF_\infty^{top}(X)$ the \emph{Ind-completed topological Fukaya category of X}. 
\end{definition}
 
By Lemma \ref{indsmallcol} Ind-completion preserves colimits, and this immediately implies the following statement.  \begin{proposition}
 \label{prop:I=Ind}
 If $X$ is a ribbon graph there is an equivalence 
 $$\cF_\infty^{top}(X)\simeq \mathrm{Ind}(\cF^{top}(X)).$$  
 \end{proposition}
  
 Proposition \ref{prop:_top} and Definition \ref{defprop:_top} indicate a key difference with the setting of small categories considered in  Section \ref{fuksmallcycsec}:  if we work with the Ind-completed (co)cyclic objects $\cI \cE^\bullet$ and 
 $\cI \cE_\bullet,$ there is no difference between Fukaya and compact Fukaya category. However, it is important to remark that  $\cF_\infty^{top} (X)$ is not equivalent in general to the Ind-completion  of $\cF_{top} (X)$. Instead $\mathrm{Ind}(\cF_{top} (X))$ is a full subcategory of $\cF_\infty^ {top} (X)$. We clarify the relationship between these two categories in Example \ref{ex:_top} below.

\begin{example}
\label{ex:_top}
Consider the ribbon graph $X$ given by a loop with no vertices. We can tabulate the value of the topological (compact) Fukaya category of $X$ and of its Ind-completed version as follows:
\begin{itemize}
\item The category $\cF^{top}(X)$ is equivalent to 
$\Perf^{(2)}(\bG_m)$. 
\item The category $\cF_{top}(X)$ is equivalent to the full subcategory of $\Perf^{(2)}(\bG_m)$ given by complexes  with compact support, $\Perf^{(2)}_{cs}(\bG_m)$. 
\item By Proposition \ref{prop:I=Ind} the category $\cF_\infty^{top}(X)$ is equivalent to $\cQ Coh^{(2)}(\bG_m)$.
\end{itemize}
Note that $\mathrm{Ind}(\cF_{top}(X)) \simeq \mathrm{Ind}(\Perf^{(2)}_{cs}(\bG_m))$ is a strict 
sub-category of $\cF_\infty^{top}(X) \simeq \cQ Coh^{(2)}(\bG_m)$. 
\end{example}

\section{The topological Fukaya category and restrictions} 
\label{restrictions closed}
In this Section we explore various  naturality properties of $\cF^{top}_
\infty$ with respect to open and closed embeddings of ribbon graphs. As first suggested by Kontsevich \cite{Ko}, the topological Fukaya category behaves like a (co)sheaf with respect to open  covers. This aspect was investigated for instance in \cite{STZ, D}. In Section \ref{sec:gluing} we will prove that, additionally, the topological Fukaya category behaves like a sheaf also with respect to certain closed covers. 

\subsection{Restriction to open subgraphs}
Let $X$ be a ribbon graph. With small abuse of notation we denote $X$ also its underlying topological space. We say that $Y \subset X$ is a \emph{subgraph} if it is a  subspace having the property that, if the intersection of $Y$ with an edge $e$ of $X$ is not empty or a vertex, then $e$ is contained in $Y$. If $Y$ is a subgraph of $X$, then it is canonically the underlying topological space of a ribbon graph, which we also denote $Y$.   Note that if $U$ and $V$ are open subgraphs of $X$, then their intersection $U \cap V$ is also an open subgraph of $X$. 

\begin{proposition}
\label{prop:open1}
Let $X$ be a ribbon graph and let 
$U \subset X$ be an open subgraph. Then there are  corestriction functors 
\begin{itemize}
\item $
C_U: \cF^{top}(U) \rightarrow \cF^{top}(X)
$
\item 
$C_{\infty, U} : \cF_\infty^{top}(U) \rightarrow \cF_\infty^{top}(X)
$
\end{itemize}
When the target of the corestriction functors is not clear from the context we will use the notations $C_U^X$ and $C^X_{\infty, U}$. 
\end{proposition}
\begin{proof}
  The construction of the functor $C_{U}$ is a formal consequence of the definition of $\cF^{top}(X)$ as a colimit. The definition $\cF^{top}(X) = \cF^{\cE}(\cL(X))$ expresses $\cF^{top}(X)$ as a colimit of dg categories indexed by the vertices and edges of $X$. In the same way, $\cF^{top}(U)$ is defined as a colimit of dg categories indexed by the vertices and edges of $U$, which are a subset of those of $X$. Thus the colimit diagram defining $\cF^{top}(U)$ is a subdiagram of the one defining $\cF^{top}(X)$, and $C_{U}$ is the resulting map between the colimits. The functor $C_{\infty,U}$ is obtained by Ind-completion. 
  See also \cite{D} Section 4 for a treatment of these corestriction functors. 
\end{proof}
\begin{remark}
By Proposition \ref{prop:I=Ind} 
$\, \cF_\infty^{top}(U)
$ and $\cF_\infty^{top}(X)$ are equivalent to the Ind-completions 
of $\cF^{top}(U)$ and $\cF^{top}(X)$. 
There is a natural equivalence  
$$
C_{\infty, U} \simeq \mathrm{Ind}(C_U)  : \cF_\infty^{top}(U) \rightarrow \cF_\infty^{top}(X). 
$$ 
In particular the corestriction $C_{\infty, U}$ is a morphism in $\mathrm{DGCat}^{(2)}_{\mathrm{cont}}$. 
\end{remark}

\begin{definition}
\label{prop:open2}
Let $X$ be a ribbon graph and let 
$U \subset X$ be an open subgraph. Then we define \emph{restriction functors}:  
\begin{itemize}
\item 
By Proposition \ref{proppropdual} there is a  natural equivalence between $\cF_{top}(-)$ and the dual of 
$\cF^{top}(-)$. The restriction  
$R^U: \cF_{top}(X) \rightarrow \cF_{top}(U)$ is obtained by dualizing $C_U$.  
\item $R_{\infty}^U : \cF_\infty^{top}(X) \rightarrow \cF_\infty^{top}(U)$ is the right adjoint of the corestriction $C_{\infty, U}$. 
\end{itemize}
When the target of the restriction functors is not clear from the context we will use the notations $R^ U_X$ and $R_{\infty, X}^ U$. \end{definition}

\begin{remark}
\label{rem:open}
Note that the functor  
$R_{\infty}^U$ cannot be realized as  $\mathrm{Ind}(R^U)$. In fact, as showed by Example \ref{ex:_top}, in general   the functors $R_{\infty}^U$ and 
$\mathrm{Ind}(R^U)$ are going to have different source and target. 
\end{remark}

\begin{remark}
Let $X$ be a ribbon graph. 
Let $V \subset U \subset X$ be open subgraphs. Then the following diagram commutes  
$$
\xymatrix{
&  \cF^{top}_\infty(U) \ar[rd]^{R^{V}_\infty} & \\
\cF^{top}_\infty(X) \ar[rr]^{R^V_\infty} \ar[ru]^{R^ U_\infty} && \cF^{top}_\infty(V).
}
$$ 
\end{remark}

\begin{proposition}
\label{prop:open3}
Let $X$ be a ribbon graph and let 
$U$ and $V$ be open subgraphs such that $X = U \cup V$. 
\begin{itemize}
\item The following is a  push-out 
in $\mathrm{DGCat}^{(2)}_{\mathrm{small}}$ 
$$
\xymatrix{
\cF^{top}(U \cap V) \ar[r]^ {C_{U \cap V}} \ar[d]_{C_{U \cap V}} & \cF^{top}(U) \ar[d]^ {C_U} \\
\cF^{top}(V) \ar[r]^ {C_V} & \cF^{top}(X).}
$$
\item The following is a  push-out  
in $\mathrm{DGCat}^{(2)}_{\mathrm{cont}}$  
$$
\xymatrix{
\cF^{top}_\infty(U \cap V) \ar[r]^- {C_{\infty, U \cap V}} \ar[d]_{C_{\infty, U \cap V}} & \cF^{top}_\infty(U) \ar[d]^ {C_{\infty, U}} \\
\cF^{top}_\infty(V) \ar[r]^ {C_{\infty, V}} & \cF^{top}_\infty(X).}
$$
\end{itemize}
\end{proposition}
\begin{proof}
The first part of the claim can be proved in the same way as Proposition 4.2 of \cite{D}. 
The second part  follows because, by Lemma \ref{indsmallcol}, the Ind-completion commutes with colimits.  
\end{proof}

\begin{proposition}
\label{prop:open4}
Let $X$ be a ribbon graph and let 
$U$ and $V$ be open subgraphs such that $X = U \cup V$. 
\begin{itemize}
\item The following is a fiber product  
in $\mathrm{DGCat}^{(2)}_{\mathrm{small}}$ 
$$
\xymatrix{
\cF_{top}(X) \ar[r]^ {R^ U} \ar[d]_{R^ V} & \cF_{top}(U) \ar[d]^ {R^{U \cap V}} \\
\cF_{top}(V) \ar[r]^- {R^{U \cap V}} & \cF_{top}(U \cap V).}
$$
\item The following is a  fiber product  
in $ \mathrm{DGCat}^{(2)}_{\mathrm{cont}}$ $$
\xymatrix{
\cF^{top}_\infty(X) \ar[r]^ {R_{\infty}^{U}} \ar[d]_{R_{\infty}^{V}} & \cF^{top}_\infty(U) \ar[d]^ {R_{\infty}^{ U \cap V}} \\
\cF^{top}_\infty(V) \ar[r]^ {R_{\infty}^{ U \cap V}} & \cF^{top}_\infty(U \cap V).}
$$
\end{itemize}
\end{proposition}
\begin{proof}
The claim follows by dualizing the push-outs in Proposition \ref{prop:open3}. 
\end{proof}

 \begin{remark}
The second diagram of Proposition \ref{prop:open4} has very different formal properties from the second diagram of Proposition \ref{prop:open3}. If $U \subset X$ is an open subgraph the corestriction $C_{\infty, X}$ preserves compact objects , but in general this is not the case for the restriction $R_\infty^ U$ (see Remark \ref{rem:open}). Thus (in general) we cannot apply $(-)^\omega$ to the second diagram of Proposition \ref{prop:open4} and obtain a diagram of small categories. 
\end{remark}

Proposition \ref{prop:open3} and Proposition \ref{prop:open4} can be 
extended in the usual way to account for arbitrary open covers of $X$: given any open cover of $X$, the (Ind-completed) (compact) Fukaya category can be realized as the homotopy (co)limit of the appropriate \v{C}ech diagram of local sections. This clarifies that  this formalism is indeed an implementation of the idea that the Fukaya category of a punctured surface should define either a sheaf or a cosheaf of categories on its spine.  

\subsection{Restriction to closed subgraphs}
\label{sec:closed}
In this Section we turn our attention to closed subgraphs and closed covers of ribbon graphs. In the context of the topological Fukaya category, restrictions to closed subgraphs have also been investigated by Dyckerhoff \cite{D}. To avoid producing here parallel arguments we will  refer to the lucid treatment contained in Section 4 of \cite{D}.

\begin{definition}
\label{def:closed}
Let $X$ be a ribbon graph. 
\begin{itemize}
\item An open subgraph $U$ of $X$ is \emph{good} if its complement $X - U$ does not have vertices of valency one. 
\item A closed subgraph $Z$ of $X$ is \emph{good} if it is the complement in $X$ of a good open subgraph.   
\end{itemize}
\end{definition}

We introduce next restriction functors 
of 
$\cF^{top}$ and $\cF^{top}_\infty$ to good closed subgraphs: we will sometimes refer to these as \emph{exceptional restrictions}, in order to mark their difference from the (co)restrictions to open subgraphs that were 
discussed in the previous Section.  

\begin{proposition}
\label{prop:closed}
Let $X$ be a ribbon graph and let 
$Z \subset X$ be a good closed subgraph. Then there are \emph{exceptional restriction} functors 
\begin{itemize}
\item $
S^Z: \cF^{top}(X) \rightarrow \cF^{top}(Z)
$
\item 
$S_\infty^Z: \cF_\infty^{top}(X) \rightarrow \cF_\infty^{top}(Z). 
$
\end{itemize}
When the source of the exceptional restriction functors is not clear from the context we will use the notations $S_X^Z$ and $S^Z_{\infty, X}$. 
\end{proposition}
\begin{proof}
  Our definition of $S^Z$ follows Section 4 of \cite{D}.
  Let $Z$ be a closed subgraph, $U = X - Z$, and consider the open subgraph $V$ consisting of $Z$ together with all half-edges of $X$ that touch $Z$. Then $X = U \cup V$ is an open covering, and by Proposition \ref{prop:open3} we have a push-out square
  $$
  \xymatrix{
    \cF^{top}(U \cap V) \ar[r]^ {C_{U \cap V}} \ar[d]_{C_{U \cap V}} & \cF^{top}(U) \ar[d]^ {C_U} \\
    \cF^{top}(V) \ar[r]^ {C_V} & \cF^{top}(X).}
  $$

  Next, we may consider another graph $\overline{V}$ that is obtained from $V$ by adding a 1-valent vertex to each edge of $V$ that does not belong to $Z$. There is a covering $\overline{V} = V \cup \Psi$, where $\Psi$ is the disjoint union of several copies of the graph with one vertex and one half-open edge. Another application of Proposition \ref{prop:open3} yields a push-out square
   $$
  \xymatrix{
    \cF^{top}(\Psi \cap V) \ar[r]^ {C_{\Psi \cap V}} \ar[d]_{C_{\Psi \cap V}} & \cF^{top}(\Psi) \ar[d]^ {C_\Psi} \\
    \cF^{top}(V) \ar[r]^ {C_V} & \cF^{top}(\overline{V}).}
  $$

  Next, we make two observations: one is that $\Psi \cap V$ naturally contains $U \cap V$, the second is that $\cF^{top}(\Psi) \simeq 0$. This implies that the first push-out diagram maps to the second, and so there is a map $\cF^{top}(X) \to \cF^{top}(\overline{V})$. Lastly we use the fact that $\cF^{top}(\overline{V}) \simeq \cF^{top}(Z)$, since $Z$ is obtained from $\overline{V}$ by contracting edges. The resulting functor is $S^{Z} : \cF^{top}(X) \to \cF^{top}(Z)$.  The functor $S^Z_\infty$ is the Ind-completion of $S^Z$. 
\end{proof}

\begin{remark}
\label{rem:closed1}
The property that the closed subgraph $Z$ is good is not 
strictly necessary to define exceptional restrictions. However this assumption allows for a somewhat simpler exposition, and it is essential in Theorem \ref{gluing-main} in the next section. We refer the reader to \cite{D} for a treatment of exceptional restrictions which does not impose additional requirements on the closed subgraphs. 
\end{remark}

\begin{proposition}
\label{prop:closed1}
Let $X$ be a ribbon graph, let $U \subset X$ be a good open subset and let $Z = X - U$. The following are cofiber sequences in $\mathrm{DGCat}^{(2)}_{\mathrm{small}}$ and in 
$\mathrm{DGCat}^{(2)}_{\mathrm{cont}},$ respectively 
$$
\cF^{top}(U) \stackrel{C_X}{\rightarrow} \cF^{top}(X) 
 \stackrel{S^Z}{\rightarrow} 
  \cF^{top}(Z), \quad 
  \cF^{top}_\infty(U) \stackrel{C_{\infty, X}}{\rightarrow} \cF^{top}_{\infty}(X) 
 \stackrel{S^Z_\infty}{\rightarrow} 
  \cF^{top}_\infty(Z)
$$
\end{proposition}
\begin{proof}
  This is Proposition 4.9 of \cite{D}.   
\end{proof}
\begin{definition}
Let $X$ be a ribbon graph and let 
$Z \subset X$ be a good closed subgraph. Then we denote  
$$
T_\infty^X: \cF_\infty^{top}(Z) \rightarrow \cF_\infty^{top}(X)  
$$ 
the right adjoint of the exceptional restriction functor. We will call it the 
\emph{exceptional corestriction} functor. When the source of the exceptional corestriction functor is not clear from the context we will use the notation   $T^X_{\infty, Z}$. 
\end{definition}

Proposition \ref{prop:compr} is a compatibility statement that relates the various restrictions  that we  have introduced so far, and it will be useful in the next section. 

\begin{proposition}
\label{prop:compr}
Let $X$ be a ribbon graph. 
Let $U \subset X$ be an open subgraph and let $Z \subset X$ be a good closed subgraph. If $Z$ is contained in $U$ then the following diagram commutes  
$$
\xymatrix{
&  \cF^{top}_\infty(U) \ar[rd]^{S^Z_{\infty, U}} & \\
\cF^{top}_\infty(X) \ar[rr]^{S^Z_{\infty, X}} \ar[ru]^{R^ U_\infty} && \cF^{top}_\infty(Z).
}
$$ 
 
 \end{proposition}

 
Before proving Proposition \ref{prop:compr}  we  introduce some preliminary notations and results. 

\begin{definition}
Let $x$ be a vertex of $X$. 
We denote:
\begin{itemize}
\item $U_x$, the smallest open subgraph of $X$ containing   the vertex $x$
\item $K_x$, the open subgraph of $X$ given by $X - x$  
\item $U_x ^p$, the intersection $U_x \cap K_x$ (the superscript $p$ stands for  \emph{punctured} neighborhood) 
\end{itemize} 
\end{definition}

\begin{definition}
Let $F: A \rightarrow \cF^{top}_\infty(K_x)$ be a  functor. We say that $A$ is \emph{not supported on} $x$ if the composite
$$
A \stackrel{F}{\longrightarrow} 
\cF^{top}_\infty(K_x) \stackrel{R_{\infty}^{U_x ^ p}}{\longrightarrow} \cF^{top}_\infty(U_x ^ p)
$$  
is equivalent to the zero functor. 
\end{definition}

\begin{lemma}
\label{lem:fff}
Let $x$ be a vertex of $X$, and let $F: A \rightarrow  \cF^{top}_\infty(K_x)$ be a functor.  If $A$ is not supported on $x$, then there is 
a natural equivalence  
$$
R^{K_x}_{\infty, X} \circ C^{X}_{\infty, K_x} \circ F \simeq F.
$$ 
\end{lemma}

\begin{proof}
We fix first some notations. If $\Gamma$ 
is a ribbon graph, and $W$ is a subset of the vertices of $\Gamma$  we set 
\begin{equation}
\label{eq gamma}
\Gamma(W) := \coprod_{w \in W} U_w, \quad \Gamma^2(W)   := \coprod_{w, w' \in W} U_w \cap U_{w'}. 
\end{equation}
Also if $L$ and $L'$ are 
objects in $\cF^{top}_\infty(\Gamma)$, we  
denote 
their Hom-complex  
$Hom_{\Gamma} (L, L')$.

Let us now return to the statement of the lemma. 
Let $V$ be the set of vertices of $X$.  
By Proposition \ref{prop:open4} 
the Ind-complete topological Fukaya category $\cF^{top}_\infty (X)$ is naturally equivalent to the equalizer\footnote{This holds only if $X(V)$ is an open cover of $X$. Note that we can always achieve this by subdividing the edges of $X$. Here and in the sequel we will assume without loss of generality that the edges of $X$ are sufficiently subdivided.} 
\begin{equation}
\label{eqqqqq}
\xymatrix
{
\cF^{top}_\infty (X) \ar[r]^-{\prod  R^{X_v}_\infty}
& 
\Big( \prod_{v \in V} 
\cF^{top}_\infty(U_v)  
\ar@<-.5ex>[r]_-*!/d0.5pt/{\labelstyle r_2  
}
\ar@<.5ex>[r]^ -*!/u0.7pt/{\labelstyle r_1
} 
& \prod \cF^{top}_\infty(U_{v_1} \cap U_{v_2}) \Big)  
}
\end{equation}
in  
$\mathrm{DGCat}^{(2)}_{\mathrm{cont}}$,  
where $r_1$ and $r_2$ are products of restriction functors. Using the notations introduced in (\ref{eq gamma}),  we can rewrite this equalizer as
$$
\xymatrix
{
\cF^{top}_\infty (X) \ar[r]^-{\prod  R^{X_v}_\infty}
& 
\Big( \cF^{top}_\infty(X(V)) 
\ar@<-.5ex>[r]_-*!/d0.5pt/{\labelstyle r_2  
}
\ar@<.5ex>[r]^ -*!/u0.7pt/{\labelstyle r_1
} 
& \cF^{top}_\infty (X^2(V)) 
\Big). 
}
$$

 Similarly if $W = V -\{x\}$ is 
the set of vertices of $K_x$, 
we obtain an equalizer diagram.
\begin{equation}
\label{eeqq}
\xymatrix
{
\cF^{top}_\infty (K_x) \ar[r]^-{\prod  R^{X_v}_\infty}
& 
\Big( 
\cF^{top}_\infty(X(W)) 
\ar@<-.5ex>[r]_-*!/d0.5pt/{\labelstyle r'_1}
\ar@<.5ex>[r]^ -*!/u0.7pt/{\labelstyle r'_2} 
&   \cF^{top}_\infty(X^2(W)) \Big).
}
\end{equation}
The inclusion  
$W \subset V$ gives projections $P$ and 
$Q$ that fit in a morphism of 
diagrams 
\begin{equation}
\label{eqqq}
\begin{gathered}
\xymatrix
{
\Big( 
\cF^{top}_\infty(X(V)) \ar[d]_P
\ar@<-.5ex>[r]_-*!/d0.5pt/{\labelstyle r_2  
}
\ar@<.5ex>[r]^ -*!/u0.7pt/{\labelstyle r_1
} 
& \cF^{top}_\infty(X^2(V)) \ar[d]^Q \Big)
\\
\Big(  
\cF^{top}_\infty(X(W))  
\ar@<-.5ex>[r]_-*!/d0.5pt/{\labelstyle r'_2}
\ar@<.5ex>[r]^ -*!/u0.7pt/{\labelstyle r'_1} 
&  \cF^{top}_\infty(X^2(W))  \Big).  
}
\end{gathered}
\end{equation}
Further, the restriction 
$R^{K_x}_{\infty, X}$ coincides with the 
morphism between the 
equalizers $\cF^{top}_\infty (X)$ and 
$\cF^{top}_\infty (K_x)$ induced by  $(1)$. 
Denote $(P)^L$ and $(Q)^L$ the left  
adjoints of $P$ and $Q$. 
The functor 
$(P)^L$ is given by the obvious quasi-fully faithful embedding
$$
\cF^{top}_\infty(X(W)) \stackrel{\subset}{\longrightarrow} 
\cF^{top}_\infty(X(V)) \simeq  
\cF^{top}_\infty(X(W)) \times 
\cF^{top}_\infty(U_x), 
$$ 
and similarly for $(Q)^L$.

We will prove the lemma in two steps. First, we show that the diagram 
\begin{equation}
\label{eqqqq}
\begin{gathered}
\xymatrix
{
& \Big( 
\cF^{top}_\infty(X(V)) 
\ar@<-.5ex>[r]_-*!/d0.5pt/{\labelstyle r_2  
}
\ar@<.5ex>[r]^ -*!/u0.7pt/{\labelstyle r_1
} 
& \cF^{top}_\infty(X^2(V))  \Big)
\\
A \ar[r]^-F & \Big(  
\cF^{top}_\infty(X(W))  \ar[u]^{(P)^L}
\ar@<-.5ex>[r]_-*!/d0.5pt/{\labelstyle r'_2}
\ar@<.5ex>[r]^ -*!/u0.7pt/{\labelstyle r'_1} 
&  \cF^{top}_\infty(X^2(W)) \ar[u]_{(Q)^L}
 \Big) 
}
\end{gathered}
\end{equation}
is commutative, that is, there are natural equivalences 
$$ 
r_1 \circ (P)^L \circ F \simeq (Q)^L  
\circ r'_1 \circ F, \text{  and   } r_2 \circ (P)^L \circ F \simeq (Q)^L 
\circ r'_2 \circ F.  
$$
We remark that the commutativity does not hold if we do not precompose with $F$. It is sufficient to prove that, for all 
$v$ and $v'$ in $V$, the diagram  commutes 
after composing on the left with the restriction
$$
R^{U_v \cap U_{v'}}_\infty: \cF^{top}_\infty(X^2(V)) \rightarrow \cF^{top}_\infty(U_v \cap U_{v'}).  
$$
If both $v$ and $v'$ are different 
from $x$, then 
$$
R^{U_v \cap U_{v'}}_\infty \circ r_i \circ (P)^L  \simeq R^{U_v \cap U_{v'}}_\infty \circ (Q)^L 
\circ r'_i,  
$$
and so commutativity holds also when precomposing with $F$. Assume on the other hand that $v=x$. Then 
$$
R^{U_x \cap U_{v'}}_\infty \circ r_i \circ (P)^L \circ F \simeq 0 \simeq R^{U_x \cap U_{v'}}_\infty \circ (Q)^L 
\circ r'_i \circ F.  
$$
The first equivalence follows from the 
support assumption on $A$, and the 
second one is a consequence of the fact that  $R^{U_x \cap U_{v'}}_\infty \circ (Q)^L  = 0$. 
Thus diagram (\ref{eqqqq}) commutes  as   claimed.

The commutativity of (\ref{eqqqq}), and the universal property of the equalizer, 
give us a functor
$$
\tilde{F}: A \longrightarrow \cF^{top}_\infty(X). 
$$
Note that both $P \circ (P)^ L$ and $Q \circ (Q)^ L$  are naturally equivalent to the identity functor, and thus 
$$
R^{K_x}_{\infty, X} \circ \tilde{F} \simeq F. 
$$
The second and final step 
in the proof consists in noticing 
that $\tilde{F}$ is equivalent to   
$C^X_{\infty, K_x} \circ F$. That is, for all  
$L_A$ in the image of $F$, and $L_X$ in $\cF^{top}_\infty(X)$, there is a natural equivalence 
$$
Hom_X(\tilde{F}(L_A), L_X) \simeq 
Hom_{K_x}(L_A, R^{K_x}_{\infty, X}(L_X)),
$$
where $Hom_X(-,-)$ and $Hom_{K_x}(-,-)$ denote respectively the hom spaces in 
$\cF^{top}_\infty(X)$ and in 
$\cF^{top}_\infty(K_x)$. 
This can be checked by computing explicitly the Hom-complexes on both sides in terms of the equalizers (\ref{eqqqqq}) 
and (\ref{eeqq}), see   
Proposition 2.2 of \cite{STZ} for a similar calculation. As a consequence there is a chain of equivalences 
$$
R^{K_x}_{\infty, X} \circ 
C^X_{\infty, K_x} \circ F
\simeq R^{K_x}_{\infty, X} \circ \tilde{F} 
\simeq F
$$
and this concludes the proof. 
\end{proof}

Let $Z$ be a good closed subgraph of $X$. Recall that 
$$
T_{\infty, Z}^X  :  \cF^{top}_\infty(Z) \longrightarrow  \cF^{top}_\infty(X)
$$
is the right adjoint of $S_{\infty, Z}^ X$. 

\begin{proposition} 
\label{prop:partial}
Let $x$ be a vertex of $X$ which does not belong to $Z$. Then 
\begin{enumerate}
\item The functor 
$$
\cF^{top}_\infty(Z) \stackrel{T_{\infty, Z}^{K_x}}{\longrightarrow}  \cF^{top}_\infty(K_x).
$$  
is not supported on $x$.
\item The diagram 
$$
\xymatrix{
&  \cF^{top}_\infty(K_x)  \ar[ld]_ {C_{\infty, K_x}^ X} & \\
\cF^{top}_\infty(X)   && \cF^{top}_\infty(Z) \ar[ll]^{T_{\infty, Z}^X} 
\ar[lu]_{T_{\infty, Z}^{K_x}} 
}
$$ is commutative.
\end{enumerate}
\end{proposition}
\begin{proof} 
It follows from Proposition \ref{prop:closed1} that   
$$
\cF^{top}_\infty(Z) \stackrel{T_{\infty, Z}^{K_x}}{\longrightarrow}  \cF^{top}_\infty(K_x) \stackrel{R_{\infty, K_x}^{K_x-Z}}{\longrightarrow}  \cF^{top}_\infty(K_x - Z)
$$
is a fiber sequence, and thus the composite is the zero functor. Since  
$U_x ^ p$ is contained in $K_x - Z$, the restriction   
$R^{U_x ^ p}_{\infty, K_x}$ factors through 
$R_{\infty, K_x}^{K_x-Z}$. This implies the first claim. As for the second claim, let us show first that there is a 
natural equivalence 
\begin{equation}
\label{eeeq}
T^{K_x}_{\infty, Z} \simeq 
R^{K_x}_{\infty, X} \circ T^{X}_{\infty, Z}. 
\end{equation}
Consider the commutative square on the right hand side of the following diagram 
$$
\xymatrix{
\cF^{top}_{\infty}(Z) \ar[rr]^-{T^ X_{\infty, Z}} \ar@{-->}[d] 
&& \cF^{top}_{\infty}(X) \ar[rr]^-{R^ {X-Z}_{\infty, X}} \ar[d]_-{R^ {K_x}_{\infty, X}} &&
\cF^{top}(X-Z) \ar[d]^-{R^ {K_x-Z}_{\infty, X}} 
\\ 
\cF^{top}_{\infty}(Z) \ar[rr]^-{T^{K_x}_{\infty, Z}} &&  
\cF^{top}_\infty(K_x) \ar[rr]^-{R^{K_x - Z}_{\infty, K_x}} && 
\cF^{top}_\infty(K_x - Z).  
}
$$
The functor induced between the fibers, which is denoted by the dashed arrow, is equivalent to the identity. This gives equivalence  (\ref{eeeq}). As a consequence 
we obtain
 $$
C_{\infty, K_x}^ X \circ T^{K_x}_{\infty, Z}  
\simeq 
C_{\infty, K_x}^ X \circ R^{K_x}_{\infty, X} \circ T^X_{\infty, Z} 
\simeq 
T^X_{\infty, Z}.
$$
Indeed, the first equivalence follows from (\ref{eeeq}) and the second one from Lemma \ref{lem:fff}. This concludes the proof. 
\end{proof}

\begin{proof}[The proof of Proposition \ref{prop:compr}] 
Let $W$ be the set of vertices of $X$ that do not belong to $U$. Note that $U$ is a connected component of the open 
subgraph $(X - W) \subset X$. By induction it is sufficient to 
prove the claim in the following two cases:
\begin{enumerate} 
\item when $U$ is equal to $K_x$ for some 
vertex $x$ of $X$, and 
\item when $U$ is a connected component of $X$. 
\end{enumerate}
Proposition \ref{prop:partial} gives a proof in   the first case. Indeed, it is sufficient 
to take right adjoints in Claim $(2)$ of 
Proposition \ref{prop:partial}  to recover the commutativity statement from Proposition \ref{prop:compr} for this class of open subgraphs.  The second case is easier. The complement  $X - U$ is open and we have a splitting 
$$
\cF^{top}_\infty(X) \simeq \cF^{top}_\infty(U) \times \cF^{top}_\infty(X-U),  
$$
and the claim follows immediately from here. 
\end{proof}


\section{The topological Fukaya category and closed covers} 
\label{sec:gluing}


In this section we prove a key gluing statement 
which is one of the main inputs in our proof of mirror symmetry for three dimensional LG models. 

Let $X$ be a connected ribbon graph whose underlying topological space is homeomorphic to a copy of $S^1$ together a finite number of open edges attached to it. We call such a ribbon graph a 
\emph{wheel}.  
Any choice of orientation on $S^1$ partitions the sets of open edges of $X$ into two subsets, which we call \emph{upward} and \emph{downward edges} respectively. An upward or downward edge is also called a \emph{spoke}. 
For our purposes it will not be important to label either of these two sets as the set of 
upward or downward   edges, but only to distinguish between the two. Thus we do not need to impose on $X$ any  additional structure beyond the ribbon structure (such as a choice of orientation).   
\begin{definition} 
Let $n_1$ and $n_2$ be in $\mathbb{Z}_{\geq 0}$. We denote $\Lambda(n_1, n_2)$ a wheel with $n_1$ upward and $n_2$ downward edges.   With small abuse of notations, we sometimes denote $\Lambda(0,0)$ simply $S^1$. 
\end{definition}

We denote $E(+)$ and $E(-)$ the open 
subgraphs of $\Lambda(n_1, n_2)$ 
given by the collection of the $n_1$   upward edges, and of the $n_2$  
downward edges respectively. 

\begin{remark}
The  notation $\Lambda(n_1, n_2)$ does not pick out a single ribbon graph, but rather a class of ribbon graphs. Indeed specifying the number of upward and downward edges does not suffice to pin down a homeomorphism type, or even the number of vertices. However all graphs of type $\Lambda(n_1, n_2)$ deform into one another in a way that does not affect the sections of $\cF^{top}$ and $\cF^{top}_\infty$ (see \cite{DK}). We use  
$\Lambda(n_1, n_2)$ to refer to 
any  ribbon graph having the properties listed in the definition.  
\end{remark}

The category $\cF_\infty^{top}(\Lambda(n_1,n_2))$ can be 
described explicitly in terms of quiver representations. Consider the closed subgraph
$$
S \subset \Lambda(n_1, n_2)
$$
where $S$ is the central circle of the wheel. The graph $S$ has $n_1 + n_2$ bivalent vertices, which are in canonical bijection with the spokes of $\Lambda(n_1, n_2)$, and its underlying topological space is $S^1$. 
Label the vertices of $S$ with  $+$ or $-$  depending on whether the corresponding spoke is upward or downward. We choose an orientation on $S$. An orientation determines a cyclic order on the edges of $S$. If $e$ is edge we denote $\tau(e)$ the edge that follows it in the cyclic order. There is a (unique) vertex of $S$ incident to both $e$ and $\tau(e)$: we say that the pair $e, \tau(e)$ is right-handed if this vertex is labeled by a 
$+$, and left-handed if it is labeled by a 
$-$.

Let $Q(n_1, n_2)$ be the quiver  defined as follows
\begin{itemize}
\item The set of vertices of $Q(n_1, n_2)$ is the set the edges of $S$. 
If $e$ is an edge of $S$, we denote $v_{e}$   the corresponding vertex of $Q(n_1, n_2)$. 
\item There is an arrow joining $v_{e}$ and $v_{\tau(e)}$. It is oriented from  $v_e$ to $v_{\tau(e)}$ if the pair $e, \tau(e)$ is right-handed, and  from $v_{\tau(e)}$ to $v_e$ otherwise. 
\end{itemize}

Recall that we can attach to 
a $\bZ$-graded dg category a $\bZ_2$-graded category by $\bZ_2$-periodization, see Section 1.2  of \cite{DK} for more details. 
Denote $Rep^{\infty}(Q(n_1, n_2))$ the $\bZ_2$-periodization of the triangulated dg category of 
(non-necessarily finite dimensional) representations of $Q(n_1, n_2)$,  $Rep^\infty(Q(n_1, n_2)) \in  \mathrm{DGCat}^{(2)}_{\mathrm{cont}}$. 
\begin{lemma} 
\label{ququiver}
There is an equivalence 
$$
\cF^{top}_\infty(\Lambda(n_1, n_2)) \simeq Rep^\infty(Q(n_1, n_2)).
$$
\end{lemma}
\begin{proof}
  This can be seen by appealing to the description of the Fukaya category of a surface provided by \cite{HKK}, which involves giving explicit presentations of the category as representations of quivers associated to certain collections of arcs on the surface. In \cite[Section 3.6]{HKK}, it is shown that this definition of the category coincides with the cosheaf-of-categories approach used in this paper. The desired quiver description corresponds to choosing a certain collection of arcs as follows. The graph $\Lambda(n_{1},n_{2})$ embeds in to a cylinder $[0,1] \times S^{1}$, in such a way that the upward edges end on $\{1\} \times S^{1}$ and the downward edges end on $\{0\}\times S^{1}$. The dual family of arcs is obtained by choosing one arc connecting the two boundary components $\{0\}\times S^{1}$ to $\{1\} \times S^{1}$ passing between each consecutive pair of edges. A representative example is depicted in Figure \ref{fig:hkk}. The dashed lines are the ribbon graph, while the dotted lines are the dual collection of arcs. Each arc corresponds to a vertex of the quiver, and the arrows in the quiver correspond to the arrows shown in the figure. The quiver corresponding to this collection of arcs is then nothing but $Q(n_{1},n_{2})$.
\begin{figure}[h]
  \centering
  \includegraphics[width=0.3\textwidth]{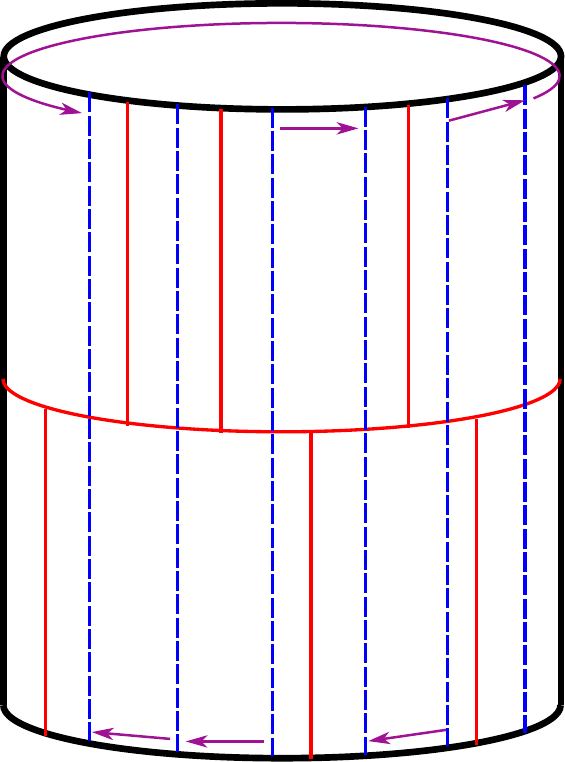}
  \caption{A ribbon graph of class $\Lambda(3,3)$ (solid), and the dual collection of arcs (dashed). The corresponding quiver has one vertex for each arc, and arrows that correspond to those shown.}
  \label{fig:hkk}
\end{figure}
\end{proof}

\begin{remark}
  In the preceding proof, we have used an equivalence between the model of the Fukaya category used in \cite{HKK} and the one used in the present paper. One may ask how this comparison is justified technically, since \cite{HKK} is written in terms of the model category of $A_{\infty}$ categories and the present paper uses an $\infty$-category dg categories (a dg category being the special case of an $A_{\infty}$ category with vanishing higher products). Fortunately, the statement of \cite[Theorem 3.1]{HKK} asserts that the Fukaya category as defined in \cite{HKK}, which is an $A_{\infty}$ category, is Morita equivalent to a homotopy colimit of a diagram of dg categories (not $A_{\infty}$ categories). Further, homotopy colimits in the model category of dg categories compute $\infty$-categorical colimits in the $\infty$-category obtained by localizing at weak equivalences; see \cite[Section 4.2.4]{Lu} where such a comparison is made. 
\end{remark}


\begin{lemma}
\label{lem:local}
The following is a fiber product in 
$\mathrm{DGCat}^{(2)}_{\mathrm{cont}}$
$$
\xymatrix{
\cF^{top}_{\infty}(\Lambda(n_1, n_2)) 
\ar[rr]^ - {S^{\Lambda(n_1, 0)}_\infty} \ar[d] _ - {S^{\Lambda(0, n_2)}_\infty} & & 
\cF^{top}_{\infty}(\Lambda(n_1, 0)) \ar[d] ^ - {S^{\Lambda(0, 0)}_\infty}\\
\cF^{top}_{\infty}(\Lambda(0, n_2)) \ar[rr] ^ - {S^{\Lambda(0, 0)}_\infty} &  &
\cF^{top}_{\infty}(\Lambda(0, 0)). 
}
$$
\end{lemma}
\begin{proof} 
%
The claim could be proved directly, but it is instructive to give  a proof based on mirror symmetry. 
Indeed, this 
clarifies the connection between gluing along closed subskeleta and Zariski descent.  Let $\mathbb{P}^1(n_1, n_2)$ be the projective line with two stacky points at $0$ and $\infty$ having isotropy isomorphic  to the groups of roots of unity $\mu_{n_1}$ and 
$\mu_{n_2}$. More formally, $\mathbb{P}^1(n_1, n_2)$ is the push-out of the following diagram 
in the category of DM stacks, 
$$
\xymatrix{
& \mathbb{G}_m \ar[dl] \ar[dr] & \\ 
[\mathbb{A}^1/\mu_{n_1}] & & [\mathbb{A}^1/\mu_{n_2}],  
}
$$
where $[\mathbb{A}^1/\mu_{n_1}]$ and 
$[\mathbb{A}^2/\mu_{n_2}]$ are the quotient stacks of $\mathbb{A}^1$ under the canonical 
action of $\mu_{n_1}$ and $\mu_{n_2}$. Zariski descent implies that the diagram
\begin{equation}
\label{eq1}
\begin{gathered}
\xymatrix{
\cQ Coh^{(2)}(\mathbb{P}^1(n_1, n_2)) \ar[r] \ar[d] & \cQ Coh^{(2)}([\mathbb{A}^1/\mu_{n_1}]) \ar[d] \\ 
\cQ Coh^{(2)}([\mathbb{A}^1/\mu_{n_2}]) \ar[r] & \cQ Coh^{(2)}(\mathbb{G}_m), 
}
\end{gathered}
\end{equation}
where all the arrows are pullbacks, is a fiber product. It follows from \cite{STZ} and \cite{DK} that diagram 
(\ref{eq1}) is in fact equivalent to the 
diagram in the statement of the lemma. More precisely, there are commutative diagrams 
$$
\xymatrix{
\cF^{top}_{\infty}(\Lambda(n_1, n_2)) \ar[d]_-\simeq \ar[r]^ -{S_\infty^{\Lambda(n_1,0)}} & 
 \cF^{top}_{\infty}(\Lambda(n_1, 0)) \ar[d]^ -\simeq \\   
\cQ Coh^{(2)}(\mathbb{P}^1(n_1, n_2)) \ar[r]  & \cQ Coh^{(2)}([\mathbb{A}^1/\mu_{n_1}]), 
}
\xymatrix{
 \cF^{top}_{\infty}(\Lambda(n_1, n_2)) \ar[d]_-\simeq \ar[r] ^ -{S_\infty^{\Lambda(n_2 ,0)}} & 
 \cF^{top}_{\infty}(\Lambda(0, n_2))  \ar[d] ^ -\simeq \\
\cQ Coh^{(2)}(\mathbb{P}^1(n_1, n_2)) \ar[r]  & \cQ Coh^{(2)}([\mathbb{A}^1/\mu_{n_2}]),
}
$$
and 
$$
\xymatrix{
\cF^{top}_{\infty}(\Lambda(n_1, 0)) \ar[d]_-\simeq \ar[r] ^ -{S_\infty^{\Lambda(0,0)}} & 
 \cF^{top}_{\infty}(\Lambda(0, 0)) \ar[d]^ -\simeq \\ 
\cQ Coh^{(2)}([\mathbb{A}^1/\mu_{n_1}]) \ar[r]  & \cQ Coh^{(2)}(\mathbb{G}_m), 
}
\xymatrix{   \cF^{top}_{\infty}(\Lambda(0, n_2)) \ar[d]_-\simeq \ar[r] ^ -{S_\infty^{\Lambda(0,0)}} & 
 \cF^{top}_{\infty}(\Lambda(0, 0))  \ar[d] ^ -\simeq \\ 
\cQ Coh^{(2)}([\mathbb{A}^1/\mu_{n_2}]) \ar[r]  & 
\cQ Coh^{(2)}(\mathbb{G}_m).}
$$
such that all vertical arrows are equivalences. 
Since diagram (\ref{eq1}) is a fiber product we 
conclude that also the diagram in the statement of 
the lemma is a fiber product. 

%
\end{proof}

The following is the main result of this 
section. In order to avoid cluttering the diagrams, we  denote  restrictions and exotic restrictions simply $R_\infty$ and $S_\infty$.

Let $X$ be a ribbon graph, and let $Z$ be a closed subgraph. It is useful to consider a combinatorial analogue of a tubular neighborhood of $Z$ inside 
$X$, which we denote $N_ZX$: the graph $N_ZX$ is given by $Z$ plus additional open edges for each edge in $X$ that does not lie in $Z$, but is incident to a vertex in $Z$. Here is a formal definition. We subdivide all the edges of $X$ which do not lie in $Z$, but whose endpoints lie in 
$Z$. We denote the resulting graph again $X$: from now on, every time we will consider the object 
$N_ZX$, we will assume implicitly that the edges of $X$ are sufficiently subdivided. Let $\overline {Z}^c$ be the maximal (non-necessarily good) closed subgraph of $X$ such that 
$$
Z \cap \overline{Z}^c = \varnothing. 
$$
We denote $N_ZX$ the   open subgraph 
$$
X - \overline {Z}^c \subset X
$$

\begin{theorem}
\label{gluing-main}
Let $X$ be a ribbon graph. Let $Z_1$ and $Z_2$ be good closed subgraphs such that:
\begin{itemize}
\item $Z_1 \cup Z_2 = X$
\item The underlying topological space of $Z_{1,2}:=Z_1 \cap Z_2$ is a disjoint union of circles
\item For every connected component $C$ of $Z_{1,2}$ 
the triple of ribbon graphs
$$
N_CX \cap Z_1 \, \subset \, N_C X  \, \supset \, N_CX \cap Z_2
$$ 
is isomorphic to 
$$
\Lambda(n_1, 0) \subset \Lambda(n_1, n_2) \supset 
\Lambda(0, n_2)
$$
for some $n_1, n_2 \in \mathbb{N}.$
\end{itemize}
Then the commutative diagram 
$$
\xymatrix{
\cF^{top}_{\infty}(X)  
\ar[rr]^ - {S_\infty} \ar[d] _ - {S_\infty} & & 
\cF^{top}_{\infty}(Z_1) \ar[d] ^ - {S_\infty}\\
\cF^{top}_{\infty}(Z_2) \ar[rr] ^ - {S _\infty} &  &
\cF^{top}_{\infty}(Z_{1,2}). 
}
$$
is a fiber product in 
$\mathrm{DGCat}^{(2)}_{\mathrm{cont}}$. 
\end{theorem}

We will assume for simplicity that $Z_{1,2}$ has only one connected component: the general case is proved in the same way.

Proving Theorem \ref{gluing-main} will require some preparation. 
Let $X$, $Z_1$ and $Z_2$ be as in Theorem \ref{gluing-main}, and assume that $Z_{1,2}$ has only one connected component.  By assumption $N_{Z_{1,2}}X$, $N_{Z_{1,2}}Z_1$ and $N_{Z_{1,2}}Z_2$ are all wheel-type graphs. We make the following notations:
\begin{itemize}
\item $U_1 = Z_1 \cup N_{Z_{1,2}}X$. The graph $U_1$ is an open subgraph of $X$ and $Z_1$ is a good closed subgraph of $U_1$  
\item $U_2 = Z_2 \cup N_{Z_{1,2}}X$. The graph $U_2$ is an open subgraph of $X$ and $Z_2$ is a good closed subgraph of $U_2$  
\item $U_{1,2} = U_1 \cap U_2$  
\item $U_1^ o = Z_1 - Z_{1,2},$ and $U_2^ o = Z_2 - Z_{1,2}$, where the superscript $o$ stands for \emph{open}. The graphs $U_1^ o$ and $U_2 ^ o$ are open subgraphs of $X$    
\item $U_1^ e = U_{1, 2} \cap U_1^ o$ and 
$U_2^ e = U_{1, 2} \cap U_2^ o$, where the superscript $e$ stands for \emph{edges}  
\end{itemize} 
Note that $U_{1,2}$ is equal to $N_{Z_{1,2}}X \cong \Lambda(n_1, n_2)$, and that the embeddings 
$$
U_1 ^ e \subset  U_{1,2}, \quad U_2 ^ e \subset  U_{1,2}
$$
are isomorphic to the embeddings of the spokes 
$$
E(+) \subset \Lambda(n_1, n_2), \quad E(-) \subset \Lambda(n_1, n_2). 
$$
Also, we have identifications
$
N_{Z_{1,2}}Z_1 = Z_1 \cap U_{12},$ and $ N_{Z_{1,2}}Z_2 = Z_2 \cap U_{12}. 
$

The key ingredient in the proof of Theorem \ref{gluing-main} is the following lemma. 
\begin{lemma}
\label{gluing-lemma}
All the interior 
squares in the commutative diagram 
$$
\xymatrix{
\cF^{top}_{\infty}(X) \ar[r]^-{R_\infty} \ar[d]_-{R_\infty} & \cF^{top}_{\infty}(U_1) \ar[r]^-{S_\infty} \ar[d]_-{R_\infty} &  \cF^{top}_{\infty}(Z_1) \ar[d]^ -{R_\infty} \\ 
\cF^{top}_{\infty}(U_2) \ar[r]^-{R_\infty} \ar[d]_-{S_\infty} & \cF^{top}_{\infty}(U_{1,2}) \ar[r]^-{S_\infty} \ar[d]_-{S_\infty} &  \cF^{top}_{\infty}(Z_1 \cap U_{1,2}) \ar[d]^ -{S_\infty} \\ 
\cF^{top}_{\infty}(Z_2) \ar[r]^-{R_\infty}   & \cF^{top}_{\infty}(Z_2 \cap U_{1,2}) \ar[r]^-{S_\infty}   &  \cF^{top}_{\infty}(Z_{1,2}  )  
}
$$
 are fiber products. 
\end{lemma}
\begin{proof}
Number clockwise the interior squares from one to four, starting with the top left one.  Square $1$ is a fiber product by Proposition \ref{prop:open4}. Square $3$ is a fiber product by Lemma \ref{lem:local}. Up to swapping $U_1$ with $U_2$, squares $2$  and $4$ are identical. So it is enough to prove that square $2$ is a fiber product.  
The proof consists of three  steps.

{\bf Step one}: We express all the vertices of square $2$  as fiber products. We start with the top vertices. 
Each of the following two diagrams 
\begin{equation}
\label{fiber1}
\begin{gathered}
 \xymatrix{
\cF^{top}_\infty(U_1) \ar[r]^- {R_\infty} \ar[d] _-{R_\infty} & \cF^{top}_\infty(U_{1,2}  ) \ar[d]^-{R_\infty} \\ 
\cF^{top}_\infty(U_1^ o ) \ar[r]^ -{R_\infty} & 
\cF^{top}_\infty(U_1^ e)  
}
  \xymatrix{
\cF^{top}_\infty(Z_1) \ar[r]^-{R_\infty} \ar[d]_-{R_\infty} & \cF^{top}_\infty(Z_1 \cap U_{1, 2}) \ar[d]^ -{R_\infty} \\ 
\cF^{top}_\infty(U_1^ o) \ar[r]^ -{R_\infty} & 
\cF^{top}_\infty(U_1^ e).}  
\end{gathered}
\end{equation}
is a fiber product in $\mathrm{DGCat}^{(2)}_{\mathrm{cont}}$  
by Proposition \ref{prop:open4}. Let us consider the bottom vertices next. The diagrams 
\begin{equation}
\label{fiber2}
\begin{gathered}
 \xymatrix{
\cF^{top}_\infty(U_{1,2}) \ar[r]^- {R_\infty} \ar[d] _-{R_\infty} & \cF^{top}_\infty(U_{1,2}) \ar[d]^-{R_\infty} \\ 
\cF^{top}_\infty(U_1^ e ) \ar[r]^ -{R_\infty} & 
\cF^{top}_\infty(U_1^ e)  
}
  \xymatrix{
\cF^{top}_\infty(Z_1 \cap U_{1,2}) \ar[r]^-{R_\infty} \ar[d]_-{R_\infty} & \cF^{top}_\infty(Z_1 \cap U_{1,2}) \ar[d]^ -{R_\infty} \\ 
\cF^{top}_\infty(U_1^ e) \ar[r]^ -{R_\infty} & 
\cF^{top}_\infty(U_1^ e).}  
\end{gathered}
\end{equation}
are trivially fiber products in 
$\mathrm{DGCat}^{(2)}_{\mathrm{cont}}$ since the horizontal arrows are identities, and any two parallel arrows are equal to each other. 

{\bf Step two:} The arrows in square 2 can be written in terms of morphisms between the fiber product diagrams constructed in step one. Let us focus, for instance, on the bottom horizontal map in square 2 
$$
\cF^{top}_{\infty}(U_{1,2}) \stackrel{S_\infty}{\rightarrow} \cF^{top}_{\infty}(Z_1 \cap U_{1,2}). 
$$
This map, which is indicated in the diagram below by a dashed arrow, is induced by the map of diagrams in $\mathrm{DGCat}^{(2)}_{\mathrm{cont}}$ given by the three arrows in the middle: $S_\infty$, $Id$ and $Id$, 
$$
\xymatrix{
& \cF^{top}_\infty(U_{1,2}) \ar[d] \ar[r]^-{S_\infty} & \cF^{top}_\infty(Z_1 \cap U_{1, 2}) \ar[d] & \\ 
\cF^{top}_\infty(U_{1,2}) \ar@/^25pt/@{-->}[rrr] \ar[ur] \ar[dr] & \cF^{top}_\infty (U_{1}^ e) \ar[r]^-{Id} &  \cF^{top}_\infty(U_1^ e ) & \cF^{top}_\infty (Z_1 \cap U_{1,2}) \ar[ul] \ar[dl] \\
& \cF^{top}_\infty(U_1^e) \ar[u] \ar[r]^-{Id} & \cF^{top}_\infty(U_1^ e)  \ar[u]  & 
}
$$
A similar reasoning holds also  for  the other  arrows in square $2$. 

{\bf Step three:} We complete the proof by using the fact that limits commute with limits. We have to show that square 2, which we reproduce as  diagram (\ref{eq fibprodprod}), is a fiber product. 
\begin{equation}
\label{eq fibprodprod} 
\begin{gathered}
$$
 \xymatrix{
\cF^{top}_\infty(U_1) \ar[r]^-{R_\infty} \ar[d]_ -{S_\infty} & \cF^{top}_\infty(Z_1 ) \ar[d]^-{R_\infty} \\ 
\cF^{top}_\infty(U_{1,2} ) \ar[r]^ -{R_\infty} & 
\cF^{top}_\infty(Z_1 \cap U_{1,2}) 
}
$$
\end{gathered}
\end{equation}
We can commute (\ref{eq fibprodprod}) and the fiber products constructed in step one past each other: thus, in order to show that   (\ref{eq fibprodprod}) is a fiber product, we can prove instead that the following are fiber products 
\begin{equation}
\label{fiber3}
\begin{gathered}
\xymatrix{
\cF^{top}_\infty(U_{1,2}) \ar[r]^-{S_\infty} \ar[d]_{Id} & 
\cF^{top}_\infty (Z_1 \cap U_{1,2}) \ar[d]^ {Id} \\ 
\cF^{top}_\infty(U_{1,2}) \ar[r]^-{S_\infty} & 
\cF^{top}_\infty (Z_1 \cap U_{1,2}) 
}
\xymatrix{
\cF^{top}_\infty(U_1^ o) \ar[r]^{Id} \ar[d]_{R_\infty} & 
\cF^{top}_\infty (U_1^ o) \ar[d]^{R_\infty} \\ 
\cF^{top}_\infty(U_1^ e) \ar[r] ^{Id} & 
\cF^{top}_\infty (U_1^ e)  
}
\xymatrix{
\cF^{top}_\infty(U_1^ e) \ar[r] ^ {Id} \ar[d]_{Id} & 
\cF^{top}_\infty (U_1^ e) \ar[d]  ^{Id}\\ 
\cF^{top}_\infty(U_1^ e) \ar[r]^{Id} & 
\cF^{top}_\infty (U_1^ e)   
}
\end{gathered}
\end{equation}
These  diagrams have the property that the horizontal arrows are identities, and any two parallel arrows are equal to each other: so they are  fiber products. This concludes the proof.  
\end{proof}

\begin{proof}[The proof of Proposition \ref{gluing-main}]
Note first that the diagram from the statement of  Proposition \ref{gluing-main} is the exterior square of the diagram from Lemma \ref{gluing-lemma}. Indeed by Proposition 
\ref{prop:compr} 
$$
S_\infty \simeq  S_\infty \circ  R_\infty . 
$$
By general properties of fiber products, since all the interior squares are fiber products, the exterior one is a fiber  product as well. This concludes the proof. \end{proof}

\begin{remark} 
Although Theorem \ref{gluing-main} is sufficient for our purposes, we expect that 
gluing formulas under closed covers hold  more generally. The importance of this kind of gluing formulas lies in the fact that they are powerful computational tools, and that they often correspond via mirror symmetry to Zariski descent statement for quasi-coherent sheaves and matrix factorizations (see, for instance, the proof of Lemma \ref{lem:local}). 
We will return to the problem of developing a comprehensive formalism  of gluing formulas along closed subskeleta for $\cF^{top}$, in dimension two and higher, in future work. 
 \end{remark}


\section{Tropical and Surface topology}
\label{surface-topology}

\subsection{Surface topology}
\label{sec:surface-top}
This section contains some remarks on surface topology that will be useful in later constructions.

Denote by $\Sigma_{g,n}$ an oriented surface of genus $g$ with $n$ punctures. Since the topology of these surfaces enters our discussion in a relatively coarse way, we will often draw the punctures as if they were boundaries, but strictly speaking $\Sigma_{g,n}$ is a noncompact (if $n > 0$) manifold without boundary. The surface $\Sigma_{g,n}$ has $n$ \emph{ends} corresponding to the punctures.\footnote{An \emph{end} of a topological space $X$ is a function $\epsilon$ from the set of compact subsets of $K \subset X$ to subsets of $X$, such that $\epsilon(K)$ is a connected component of $X \setminus K$, and such that if $K_1 \subset K_2$, then $\epsilon(K_2) \subset \epsilon(K_1)$. Thus ends are intrinsic to the space $X$, and make sense without reference to a compactification.}

If $\Sigma_1$ and $\Sigma_2$ are two oriented punctured surfaces, we may form a new surface by the well-known \emph{end connect sum} operation.

\begin{definition}
  Choose a puncture $p_1$ on $\Sigma_1$ and a puncture $p_2$ on $\Sigma_2$. Identify a neighborhood of $p_1$ with $S^1\times (-1,-1/2)$ and a neighborhood of $p_2$ with $S^1\times (1/2,1)$, and replace the union of these neighborhoods by a single punctured cylinder $S^1 \times (-1,1) \setminus {(1,0)}$. The result $\Sigma_1 \#_{p_1,p_2} \Sigma_2$ is called the \emph{end connect sum of $\Sigma_1$ and $\Sigma_2$ at the punctures $p_1$ and $p_2$}.  
\end{definition}

The end connect sum can also be described as attaching a one-handle to $\Sigma_1 \coprod \Sigma_2$. If alternatively we think in terms of bordered surfaces, the operation consists of adding a strip connecting two boundary components. If $\Sigma_i$ has genus $g_i$ and $n_i$ punctures ($i = 1,2$), then $\Sigma \#_{p_1,p_2} \Sigma_2$ has genus $g_1+g_2$ and $n_1+n_2-1$ punctures.

The effect of end connect sum on skeleta is straightforward.

\begin{lemma}
  Let $X_i$ be a skeleton for $\Sigma_i$ ($i=1,2$). Produce from $X_i$ a ribbon graph with a noncompact edge connecting $X_i$ to the puncture $p_i$; call the result $X_i'$. Then a skeleton for $\Sigma_1\#_{p_1,p_2} \Sigma_2$ is obtained by connecting the noncompact edges of $X_1'$ and $X_2'$ inside the attaching region.
\end{lemma}

\begin{example}
\label{example:tori-plus-sphere}
We can decompose $\Sigma_{g,n}$ ($n > 0$) into an iterated end connect sum of $\Sigma_{1,1}$ and $\Sigma_{0,2}$. Indeed, taking end connect sum of $g$ copies of $\Sigma_{1,1}$ (always summing at the unique punctures) yields a surface of type $\Sigma_{g,1}$. Taking end connect sum of $n-1$ copies of $\Sigma_{0,2}$ (summing at a single puncture of each) yields a surface of type $\Sigma_{0,n}$. Then end connect summing $\Sigma_{g,1}$ and $\Sigma_{0,n}$ yields $\Sigma_{g,n}$. By choosing skeleta for $\Sigma_{1,1}$ (consisting, say, of two loops on the torus) and for $\Sigma_{0,2}$ (say a single circle), we thus obtain a skeleton for $\Sigma_{g,n}$. This is pictured in Figure \ref{fig:example}.
\begin{figure}[h]
  \centering
  \includegraphics[width=0.8\textwidth]{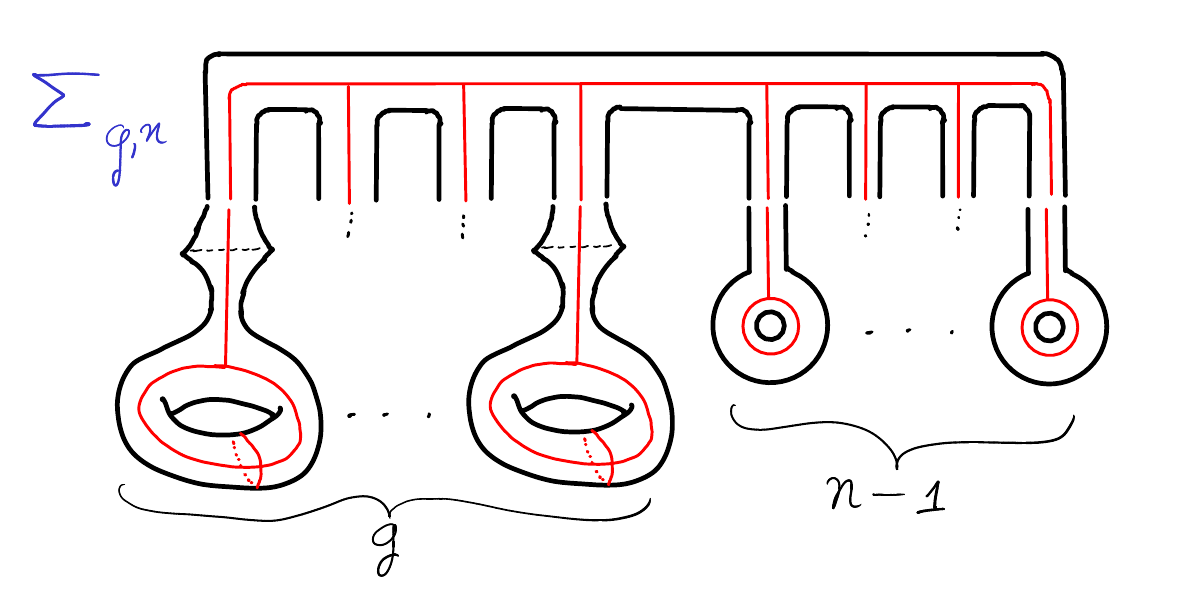}
  \caption{Decomposition of $\Sigma_{g,n}$ into end connect sum of $g$ copies of $\Sigma_{1,1}$ and $n-1$ copies of $\Sigma_{0,2}$}
  \label{fig:example}
\end{figure}
\end{example}

In this paper, we are interested in skeleta with a certain shape near the punctures.

\begin{definition}
  Let $\Sigma$ be a punctured surface, $p$ a puncture of $\Sigma$, and $X \subset \Sigma$ a skeleton for $\Sigma$. The component of $\Sigma \setminus X$ containing the puncture $p$ is topologically a punctured disk, and its boundary is a subgraph of $X$. We say that \emph{$X$ has a cycle at the puncture $p$} if this subgraph is a cycle.

If $p_1$ and $p_2$ are distinct punctures, we say that $X$ has \emph{disjoint cycles} at $p_1$ and $p_2$ if it has cycles at $p_1$ and $p_2$, and these cycles are disjoint in $X$.
\end{definition}

Model the pair of pants as $\bC - \{-2, 2\}$. If $x$ and $y$ are in $\bC$, and $\epsilon$ is a positive real number, we denote $S(x, \epsilon) \subset \bC$ the circle of center $x$ and radius $\epsilon$, and $I(x,y) \subset \bC$ the straight segment joining $x$ and $y$. We call \emph{$\Theta$ graph} the skeleton of the pair of pants given by 
$$
S(0, 3) \cup I(-3i, 3i). 
$$  
We call \emph{dumbell graph} the skeleton given by 
$$
S(-2, 1) \cup I(-1, 1) \cup S(2, 1). 
$$
The $\Theta$ graph has a cycle at each of the three punctures, whereas the dumbbell graph has a cycle at only two: for the third puncture, the boundary of the corresponding component of $\Sigma \setminus X$ consists of the entire skeleton. On the other hand, in the $\Theta$ graph the cycles for any two punctures are not disjoint, whereas in the dumbbell graph they are disjoint. These graphs are shown in Figure \ref{fig:popskeleta}.

\begin{figure}[h]
  \centering
  \includegraphics[width=0.8\textwidth]{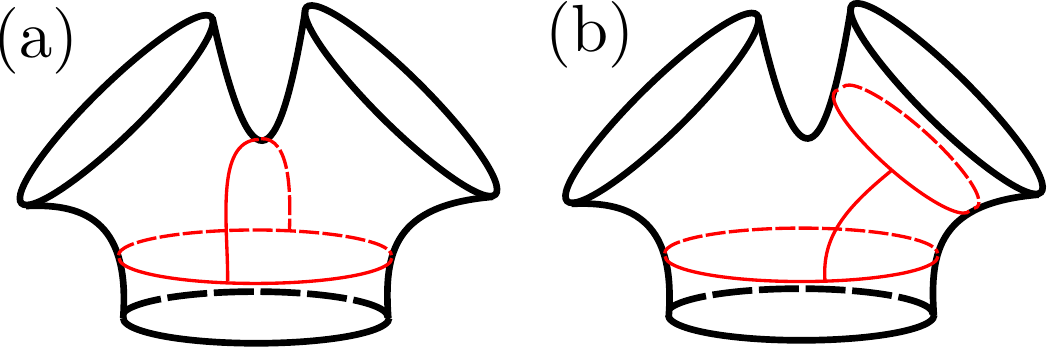}
  \caption{Two skeleta for the pair of pants. (a) $\Theta$ graph. (b) Dumbbell graph.}
  \label{fig:popskeleta}
\end{figure}

\begin{lemma}
\label{lem:circleatpuncture}
  $\Sigma_{g,n}$ admits a skeleton that has a cycle at every puncture but one.
\end{lemma}

\begin{proof}
  This is furnished by Example \ref{example:tori-plus-sphere}.
\end{proof}

In fact, whenever $X$ is a skeleton with a cycle at a particular puncture, $\Sigma$ and $X$ can be decomposed into an end connect sum in a manner similar to that of Example \ref{example:tori-plus-sphere}. 

\begin{lemma}
\label{lem:puncture-decomposition}
Let $X$ be a ribbon graph for $\Sigma$ that has a cycle at the puncture $p$. Suppose that $r$ other edges are incident to the cycle at $p$. Then $\Sigma$ can be decomposed into an end connect sum of $\Sigma'$ and $\Sigma''$, where $\Sigma'' \cong \Sigma_{0,2}$, and $\Sigma'$ has one fewer puncture than $\Sigma$, and $X$ can be decomposed into the sum of $X'$ and $X''$, where $X'$ and $X''$ are ribbon graphs embedded in $\Sigma'$ and $\Sigma''$ respectively, each with $r$ noncompact edges approaching the punctures where the connect sum is taken, and $X$ is obtained by connecting the noncompact edges of $X'$ to those of $X''$. See Figure \ref{fig:end-connect-sum}.
\end{lemma}

\begin{figure}[h]
  \centering
  \includegraphics[width=0.8\textwidth]{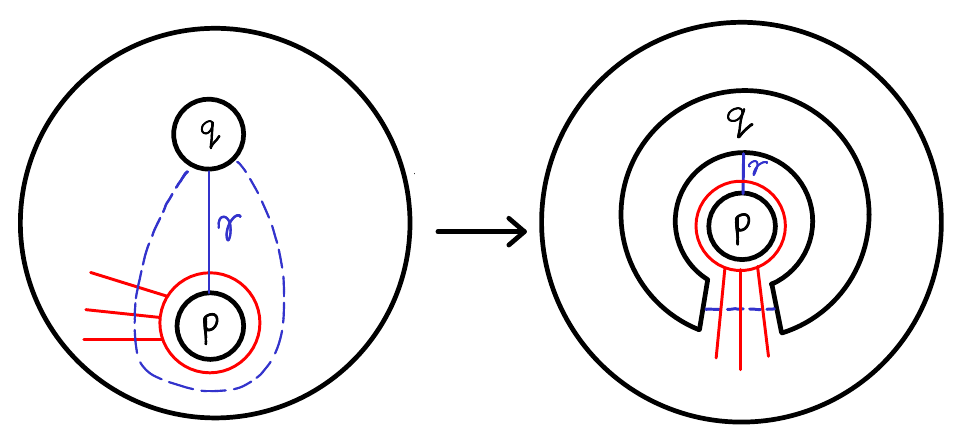}
  \caption{Decomposition of $\Sigma$ into an end connect sum, depending on a choice of path $\gamma$ between two punctures. The dashed line is for visual reference.}
  \label{fig:end-connect-sum}
\end{figure}

\begin{proof}
  The idea is to deform our picture of $\Sigma$ so that the cycle at $p$ is pulled out at another puncture $q$ of $\Sigma$. To do this, what is needed is a path $\gamma$ in $\Sigma$ from $p$ to $q$ that does not cross any other points of the skeleton $X$. But this is always possible, since every component of $\Sigma \setminus X$ is homeomorphic to a punctured disk.
\end{proof}

We remark that the proof shows that if $r$ edges are incident to the cycle at $p$, then there are essentially $r$ choices for how to decompose $\Sigma$ and $X$ as in the lemma.

\begin{lemma}
\label{lem:puncture-mod}
\begin{enumerate}
\item Let $X_1$ and $X_2$ be two ribbon graphs for $\Sigma$ that both have cycles at the puncture $p$. Then it is possible to connect $X_1$ to $X_2$ by a sequence of contractions and expansions so that every intermediate graph also has a cycle at $p$.
\item Let $\Sigma$ be a surface with at least three punctures. Let $p$ be a puncture of $\Sigma$, and let $X$  be a skeleton for $\Sigma$ that has a cycle at $p$. Let $p'$ be another puncture of $X$. It is possible to modify $X$ to $X'$ so that $X'$ also has cycles at $p$ and $p'$, and so that every intermediate graph also has a cycle at $p$. 
\end{enumerate}
\end{lemma}

\begin{proof}
  \begin{enumerate}
  \item First, if there is more than one edge incident to the cycle at $p$ in $X_1$ or $X_2$, we can apply contractions to gather together all of these edges into a single vertex of valence $r+2$, and then apply a single expansion to ensure that in both $X_1$ and $X_2$ only a single edge is incident to the cycle at $p$. None of these moves destroy the cycle at $p$ in $X_1$ or $X_2$.
    
    Let $p' \neq p$ be another puncture. Choose a path $\gamma$ from $p$ to $p'$. As in Lemma \ref{lem:puncture-decomposition}, we may assume that $\gamma$ only crosses $X_1$ at the cycle at $p$. Once this choice is made, we cannot assume the same holds true for $X_2$, so $\gamma$ will cross $X_2$ at some number of edges not contained in the cycle at $p$; let $k$ be this number. Now we apply the idea of stretching the surface from \ref{lem:puncture-decomposition}, using the chosen path $\gamma$. This decomposes $\Sigma$ into an end connect sum of $\Sigma'$ and $\Sigma''$, where $\Sigma''$ is has genus 0 and 2 punctures in such a way that the cycle at $p$ ends up in the $\Sigma''$ factor. See Figure \ref{fig:stretch}(a). 
    
    Now we consider how $X_1$ and $X_2$ look with respect to this decomposition. Since the path $\gamma$ only intersects $X_1$ at the cycle at $p$, $X_1$ decomposes just as in Lemma \ref{lem:puncture-decomposition}. On the other hand, $X_2$ is as shown in Figure \ref{fig:stretch}(a). The part of $X_2$ that ends up in $\Sigma''$ consists of a cycle at $p$ together with $k$ parallel arcs. This is connected to the rest of $X_2$ via $2k+1$ noncompact edges.

    The next step is to apply moves to $X_2$ that move the $k$ arcs out of $\Sigma''$ and into $\Sigma'$. Observe that the space between two neighboring arcs corresponds, in the summed surface $\Sigma$, to a component of $\Sigma \setminus X_2$, which is a punctured disk. Start with the outermost arc, call it $a$. Let $D$ denote the punctured disk corresponding to the region just inside $a$, so $D$ is a punctured disk. The arc $a$ ends at two vertices in $\Sigma'$. By a sequence of contractions and expansions, we may move one of the ends along the boundary of $D$, through $\Sigma''$, and back into $\Sigma'$. We can also follow the disk $D$ throughout this process. (Depending on how it is done, the puncture of $D$ may also move through $\Sigma''$.) This is depicted in Figure \ref{fig:stretch}(b). Since none of these moves destroy the cycle at $p$, this reduces us to the situation where $k = 0$. 
    
    In the case $k = 0$, we have decompositions of $X_1$ and $X_2$ into end connect sums of $X_1'$ and $X'_2$ in $\Sigma'$, each having a single noncompact edge, and $X''$ in $\Sigma''$ consisting of a single cycle with a single noncomapact edge. Now we apply the fact that any two ribbon graphs for $\Sigma'$ with a single noncompact edge asymptotic to a puncture can be connected by a sequence of moves, by a result of Harer \cite[Proposition 3.3.9]{DK}. Evidently, such moves do not destroy the cycle at the puncture in $\Sigma''$, so we are done. 
    
  \item Since the surface $\Sigma$ has at least three punctures, there is a ribbon graph $X'$ that has cycles at both $p$ and $p'$. Now apply the first part of the lemma.
  \end{enumerate}
  \begin{figure}[h]
  \centering
  \includegraphics[width=0.75\textwidth]{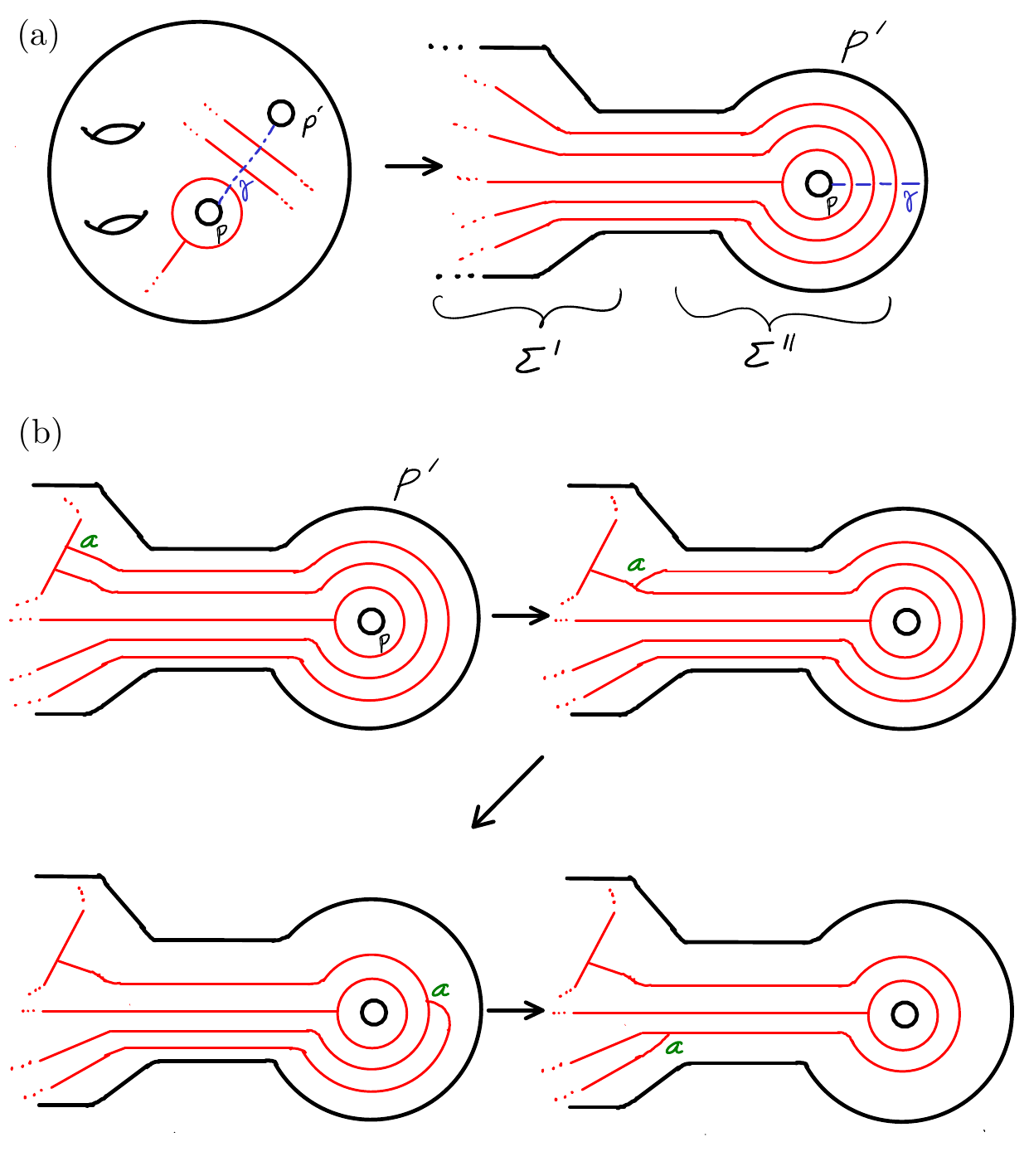}
  \caption{(a) Decomposition of surface into end connect sum, and corresponding decomposition of ribbon graph. (b) Moving an arc through $\Sigma''$. The point marked $a$ is the end of the arc that is being moved.}
  \label{fig:stretch}
\end{figure}
\end{proof}

\subsection{Tropical topology}
\label{sec:tropical-top}

Since our strategy is to prove HMS inductively by gluing together pairs of pants, and the gluings are controlled by a balanced tropical graph $G_{\mathcal{T}}$ associated to the given toric Calabi-Yau Landau-Ginzburg model $(X_{\mathcal{T}},W_{\mathcal{T}})$, we collect here some elementary remarks about the topology of such graphs that will be useful. The main point is to keep track of the non-compact edges of $G$, since these are edges where we \emph{never} need to glue in our induction; we also point out that $G$ can be built up in such a way that we never need to glue along all the edges incident to a single vertex.

Let $G$ be trivalent graph with both finite and infinite edges. For each edge $e$, we have an \emph{orientation line} $\det(e)$ that is the $\bZ$-module generated by the two orientations of $e$ modulo the relation that their sum is zero. A  \emph{(planar integral) momentum vector} $p_e$ on the edge $e$ is a linear map $p_e : \det(e) \to \bZ^2$.
\begin{definition}
  A pair $(G,\{p_e\}_{e\in \Edge(G)})$ consiting of a graph and a set of momenta is a \emph{balanced tropical graph} if momentum is conserved at each vertex. Namely, for each vertex $v$ of $G$, 
  \begin{equation}
    \sum_{e} p_e(\text{inward orientation}) = 0
  \end{equation}
  where the sum is over all edges $e$ incident to $v$, and the inward orientation is the one pointing toward $v$.

  Such a graph is additionally called \emph{nondegenerate} if the values of the momenta at each vertex span $\bZ^2$, or equivalently if not all momenta at a given vertex are proportional.
\end{definition}

\begin{definition}
  A \emph{planar immersion} of $(G,\{p_e\}_{e\in \Edge(G)})$ is a continuous map $i : G \to \bR^2$, such that derivative of $i$ along an edge $e$ in the direction $o$ is positively proportional to the momentum $p_e(o)$. Note that we do not require $i$ to be proper on infinite edges.
\end{definition}

From now on we will consider nondegenerate balanced tropical graphs $(G,\{p_e\}_{e\in \Edge(G)})$ with planar immersion $i$. Planar immersions of balanced tropical graphs are in some sense ``harmonic,'' so it is not surprising that they satisfy a version of the maximum principle:

\begin{lemma}
  \label{lem:atleasttwo}
  $G$ has at least two infinite edges.
\end{lemma}

\begin{proof}
  Let $i : G \to \bR^2$ be a planar immersion, and let $\pi : \bR^2 \to \bR$ be the orthgonal projection onto any given line of irrational slope. Then consider the function $\pi \circ i : G \to \bR$. Nondegeneracy implies that $p_e \neq 0$ for any $e$, so $\pi \circ i$ is not constant on any edge. 

We claim that $\pi \circ i$ does not achieve its supremum. For if it did, this would necessarily occur at a vertex, as $\pi \circ i$ is linear and nonconstant on all edges. At the vertex, the images under $i$ of all incident edges lie in the same half-plane determined by the linear function $\pi$. This is not compatible with the balancing condition, since three non-zero vectors in the same half-plane cannot sum to zero.

The same reasoning applied to $- \pi \circ i$ shows that $\pi \circ i$ does not achieve its infimum. Therefore there must be two infinite edges on which the supremum and infinimum are approached but not obtained.
\end{proof}

Now let $e_0$ be an infinite edge of $G$; it is incident to a vertex $v_0$, and there are three possibilities for the local structure of $G$ at $v_0$:
\begin{enumerate}
\item\label{case:oneinfinite} $v_0$ is incident to one infinite edge, namely $e_0$.
\item\label{case:twoinfinite} $v_0$ is incident to two infinite edges, namely $e_0$ and one other $e_1$.
\item $v_0$ is incident to three infinite edges. Then $v_0$ and these three edges comprise a connected component of $G$.
\end{enumerate}
See Figure \ref{fig:vertexcases}.
\begin{figure}[h]
  \centering
  \includegraphics[width=0.7\textwidth]{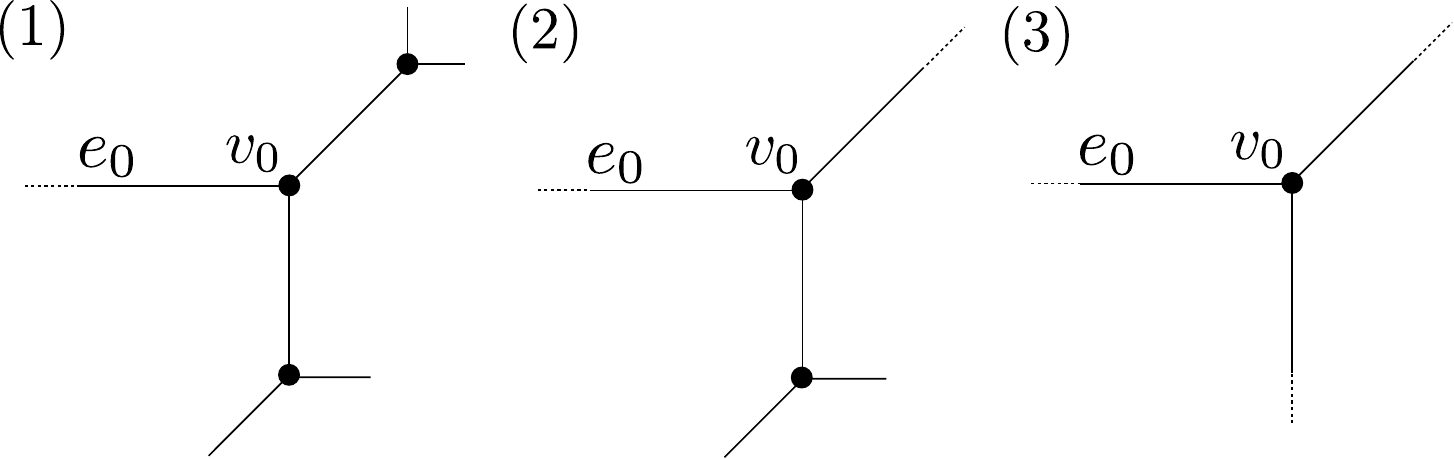}
  \caption{Cases at a vertex.}
  \label{fig:vertexcases}
\end{figure}

\begin{lemma}
\label{lem:gprime}
  In case \ref{case:oneinfinite}, let $G'$ be the graph obtained from $G$ by deleting $e_0$ and $v_0$. In case \ref{case:twoinfinite}, let $G'$ be the graph obtained from $G$ by deleting $e_0, e_1$ and $v_0$. Then $G'$ has an infinite edge not originally incident to $v_0$ (in $G$).
\end{lemma}

\begin{proof}
  In case \ref{case:twoinfinite}, this follows from Lemma \ref{lem:atleasttwo}, as $G'$ must have two infinite edges, and only one is incident to $v_0$.

In case \ref{case:oneinfinite}, let $e_1$ and $e_2$ be the other edges incident to $v_0$; these become infinite edges in $G'$. Let $\pi : \bR^2 \to \bR$ be orthgonal projection onto an irrational line chosen so that both $i(e_1)$ and $i(e_2)$ lie in the half-plane defined by the inequality $\pi(x) \geq \pi(i(v_0))$. The argument from the proof of Lemma \ref{lem:atleasttwo} shows that $\pi \circ i$ approaches its supremum along some infinite edge. This edge cannot be $e_1$ or $e_2$, as $\pi\circ i$ is \emph{decreasing} in the noncompact direction on these edges. 
\end{proof}

\begin{lemma}
\label{lem:graph-decomp}
Given $i : G \to \bR^2$ a planar immersion, there exists a sequence $i_j : G_j \to \bR^2$, $j= 1, \dots, N$ with the following properties. 
\begin{enumerate}
\item $i_j : G_j \to \bR^2$ is a planar immersion of the tropical graph $G_j$,
\item $i_N : G_N \to \bR^2$ equals $i : G \to \bR^2$.
\item There is a continuous embedding $G_j \to G_{j+1}$ such that $i_j = i_{j+1}|G_j$, and such that $G_{j+1}$ is obtained from $G_j$ by gluing a single trivalent vertex to $G_j$ along either one or two of the noncompact edges of $G_j$, and also extending some other noncompact edges of $G_j$.
\end{enumerate}
\end{lemma}

\begin{proof}
  Begin with $i : G \to \bR^2$, and once again choose a projection $\pi : \bR^2 \to \bR$. If $\pi$ is chosen generically, each fiber of $\pi \circ i$ will contain at most one vertex of $G$. Let the values of $\pi \circ i$ on the vertices be $\lambda_1,\lambda_2, \dots, \lambda_n \in \bR$. Then take $G_i = (\pi \circ i)^{-1}(-\infty,\lambda_i)$.
\end{proof}
\section{The induction}
\label{The induction}
This section contains the main induction that proves HMS.

For any oriented punctured surface $\Sigma$ equipped with a skeleton $X$, we associate the topological Fukaya category $\cF_\infty^{top}(X)$. Recall that we denote $S^1$ a ribbon graph consisting of type $\Lambda(0,0)$. For each puncture $p$ of $\Sigma$, we can define a restriction functor $$\cF_\infty^{top}(X) \to \cF_\infty^{top}(S^1)$$ 
as follows. If the graph $X$ contains a cycle corresponding to the puncture $p$, then $R$ is defined directly as an exceptional  restriction functor. If not, then $R$ is defined by first choosing another skeleton $X'$ that does have a cycle corresponding to the puncture $p$, and that is obtained from $X$ by a sequence of contractions and expansions. By Proposition \ref{prop:ribbon}, there is an equivalence $$
\Phi : \cF_\infty^{top}(X) \to \cF_\infty^{top}(X')$$
which is canonically associated with such a sequence of contractions and expansions.   Composing $\Phi$ with the closed restriction functor $S^{S^1}_\infty:\cF_\infty^{top}(X') \to \cF_\infty^{top}(S^1)$  gives the desired restriction functor. We first show that this functor does not depend on the choice of skeleton used to define it.

\begin{lemma}
  Let $X_1$ and $X_2$ be two skeleta for $\Sigma$ that both have cycles corresponding to the puncture $p$, and that are obtained from each other via a sequence of contractions and expansions. Then there is a commutative diagram
  \begin{equation}
    \xymatrix{
      \cF_\infty^{top}(X_1) \ar[r]^{\Phi} \ar[dr]_-{S_\infty^{S^1}}  & \cF_\infty^{top}(X_2) \ar[d]^-{S_\infty^{S^1}}\\
      & \cF_\infty^{top}(S^1)
    }  
  \end{equation}
  where $\Phi$ denotes the canonical equivalence, and $R$ denotes closed restriction maps.
\end{lemma}

\begin{proof}
  This is an application of Lemma \ref{lem:puncture-mod}. Since we an arrange that the contractions and expansions that implement $\Phi$ do not destroy the cycle at $p$, at every step the desired commutative diagram both makes sense and holds true.
\end{proof}

\begin{definition}
Let $p$ be a puncture of $\Sigma$. We denote
$$
R_p: \cF_\infty^{top}(X) \to \cF_\infty^{top}(S^1)
$$
the corresponding restriction functor. 
\end{definition}

By Definition \ref{defff}, for any nondegenerate balanced graph with planar immersion $G$, we have a matrix-factorization-type category $\cB(G)$. For each external edge of $G$, there is a restriction functor $\cB(G) \to \cB(E)$, where $E$ is the graph consisting of a single bi-infinite edge. We can associate to the graph $G$ a punctured Riemann surface 
 $\Sigma(G)$ in a way that generalizes the familiar correspondence between an algebraic curve and its tropicalization. Namely, each vertex of $G$ corresponds to a pair of pants, while the edges correspond to cylinders: the graph $G$ encodes the way in which the pairs of pants are glued along cylinders. Then the genus of $\Sigma(G)$ is  equal to the number of relatively compact connected components in $\bR^2-G$, and the number of punctures is given by the number of infinite edges of $G$.

Now we come to the main result, that category $\cF_\infty^{top}(X)$ is equivalent to the category $\cB(G)$ (see Definition \ref{defff}), where $X$ is a skeleton for $\Sigma(G)$. Since our method involves successively gluing pairs of pants inductively, we must include in the induction a statement on the restriction maps at the punctures.

\begin{theorem}
\label{the induction}
  If $X$ is a skeleton for $\Sigma(G)$, then there is an equivalence of categories $\Psi: \cF_\infty^{top}(X) \to \cB(G)$ with the property that for each infinite edge $e$ of $G$, and corresponding puncture $p(e)$, there is a commutative diagram
  \begin{equation}
    \xymatrix{
      \cF_\infty^{top}(X) \ar[r]^\Psi \ar[d]_{R_{p(e)}} & \cB(G)\ar[d]^-{R^{ \cB } } \\
      \cF_\infty^{top}(S^1) \ar[r]^\Psi & \cB(e)
    }
  \end{equation}
  where the vertical arrows are restriction functors.
\end{theorem}

\begin{proof}
  We may regard the graph $G$ as being constructed from a collection of trivalent vertices by gluing infinite edges to each other. By Lemma \ref{lem:graph-decomp}, there is a collection of graphs $G_i$, $i = 1,\dots, N$ such that $G_N = G$, and $G_{i+1}$ is obtained from $G_i$ by gluing a single trivalent vertex to either one or two infinite edges of $G_i$ (but not at all three edges simultaneously).

We shall prove the assertions in the theorem by induction on $i$. In the base case $i = 1$, we are simply considering the pair of pants, for which the result is known. See for instance Theorem 1.13 of \cite{N4}. 

For the induction step, the induction hypothesis is the statement of the theorem for $G_i$. In passing from $G_i$ to $G_{i+1}$, we attach a trivalent vertex $T$; correspondingly, $\Sigma(G_{i+1})$ is obtained from $\Sigma(G_i)$ by attaching a pair of pants $\Sigma(T)$. Now there are two cases, depending on whether the gluing involves one edge or two. 

\emph{Case of one edge}: Denote $e$ the edge along which $G_i$ is glued to  $T$.  Then both $\Sigma(G_i)$ and 
$\Sigma(T)$ have a puncture corresponding to $e$: we denote in the same way, namely $p(e)$, the corresponding puncture on $\Sigma(G_i)$ and the puncture 
on $\Sigma(T)$. We may choose skeleta $X$ for $\Sigma(G_i)$ and $Y$ for $\Sigma(T)$ such that both $X$ and $Y$ have a cycle, respectively, at the puncture $p(e)$ of $\Sigma(G_i)$ 
and at the puncture $p(e)$ of $\Sigma(T)$. We then have a diagram
\begin{equation}
\xymatrix{
  \cF_\infty^{top}(X) \ar[r]^{R_{p(e)}} \ar[d]_\Psi &  \cF_\infty^{top}(S^1) \ar[d]_\Psi &  \cF_\infty^{top}(Y) \ar[l]_{R_{p(e)}} \ar[d]_\Psi\\
  \cB(G_i) \ar[r]^{R^{ \cB } } & \cB(e) & \cB(T) \ar[l]_{R^{ \cB } } \\
}
\end{equation}
where the horizontal arrows are restriction functors, and the vertical arrows are the equivalences given by the induction hypothesis. The fact that both squares commute is also part of the induction hypothesis. This equivalence of diagrams implies the equivalence of  fiber products:
\begin{equation}
  \xymatrix{
    \cF_\infty^{top}(X\coprod_{S^1} Y) \ar[r] \ar@/^1.5pc/[rr]^-\Psi \ar[d] & \cF_\infty^{top}(Y) \ar[d]^-{R_{p(e)}} \ar@/^1.5pc/[rr]^-\Psi & \cB(G_i\coprod_E T) \ar[r] \ar[d] & \cB(T) \ar[d] \\
    \cF_\infty^{top}(X) \ar[r]^-{R_{p(e)}} \ar@/_1.3pc/[rr]_-\Psi &  \cF_\infty^{top}(S^1) \ar@/_1.3pc/[rr]_-\Psi & \cB(G_i) \ar[r]  & \cB(e)\\
  }
\end{equation}
In the diagram above, the squares are fiber products, and the curved arrows are equivalences of categories. In particular, since $G_{i+1} \coprod_E T = G_i$, and $X \coprod_{S_1} Y$ is a skeleton for $\Sigma(G_{i+1})$, we have an equivalence
\begin{equation}
  \Psi: \cF_\infty^{top}(\Sigma(G_{i+1})) \to \cB(G_{i+1})
\end{equation}
To complete the proof of the induction step, we must also consider the restriction functors to the punctures of $\Sigma(G_{i+1})$. On the $\cB$-side, the infinite edges $G_{i+1}$ correspond to infinite edges of $G_i$ and $T$, minus the edge  $e$ that  we glue along. For each infinite edge $e'$ of $G_{i+1}$, we have a restriction functor $R^{ \cB }_\infty : \cB(G_{i+1}) \to \cB(e')$. This functor factors through either $\cB(G_i)$ or $\cB(T)$, according to whether $e'$ comes from $G_i$ or $T$. On the $\cF$-side, we have a corresponding restriction functor $R_{p(e')} : \cF_\infty^{top}(X \coprod_{S^1} Y) \to \cF_\infty^{top}(S^1)$. Strictly speaking, the definition of this functor requires choosing a skeleton for $\Sigma(G_{i+1})$ that has a cycle at the puncture $p(e')$, and $X\coprod_{S^1} Y$ may not have this property (and furthermore it is impossible for it to have this property with respect to every puncture simultaneously). The solution is Lemma \ref{lem:puncture-mod}, which says that we can modify either $X$ or $Y$ only in order to achieve that $X\coprod_{S^1}Y$ also has a cycle at $p(e')$. Since this modification can be implemented on $X \coprod_{S^1} Y$ simultaneously, we find that the restriction to $p(e')$ factors through the the closed restriction to either $X$ or $Y$. If the puncture $p(e')$ comes from $\Sigma(G_i)$, there is therefore a commutative diagram of closed restriction functors
\begin{equation}
  \xymatrix{
    \cF_\infty^{top}(X\coprod_{S^1}Y) \ar[r]^-{S^{X}_\infty} \ar[dr]_{R_{p(e')}}
    & \cF_\infty^{top}(X) \ar[d]^{R_{p(e')}}\\
    & \cF_\infty^{top}(S^1)\\
  }
\end{equation}
In the case that $p(e')$ comes from $T$, the same diagram holds with $Y$ in place of $X$ in the upper-right node. Comparing the two sides, we have a diagram
\begin{equation}
  \xymatrix{
     \cF_\infty^{top}(X\coprod_{S^1}Y) \ar[r] 
     \ar[dr]_{R_{p(e')}} \ar@/^1.5pc/[rr]^\Psi
    & \cF_\infty^{top}(X) \ar[d]^ {R_{p(e')}} \ar[d]\ar@/^1.5pc/[rr]^\Psi 
    & \cB(G_i\coprod_E T) \ar[r] \ar[dr] 
    & \cB(G_i) \ar[d]\\
    & \cF_\infty^{top}(S^1) \ar@/^1.5pc/[rr]^\Psi
    & & \cB(e')\\
  }
\end{equation}
In this diagram, the curved $\Psi$ arrows (which are equivalences) commute form commutative squares with the horizontal and vertical arrows, and therefore they also form a commutative square with the diagonal arrows. This establishes the desired compatibility between restriction functors to infinite edges of $G_{i+1}$ with restrictions to punctures of $\Sigma(G_{i+1})$.

\emph{Case of two edges}:
Let $e_1$ and $e_2$ be the two edges along which 
$G_i$ and $T$ are glued. As before, we denote 
$p(e_1)$ and $p(e_2)$ both the two punctures on 
$\Sigma(G_i)$, and 
the two punctures on 
$\Sigma(T)$, that correspond to $e_1$ and $e_2$. Choose a skeleton $X$ for $\Sigma(G_i)$ that has disjoint cycles at the punctures $p(e_1)$ and $p(e_2)$. Choose a skeleton $Y$ for $T$ that has disjoint cycles at the punctures $p(e_1)$ and $p(e_2)$ (this $Y$ is necessarily a dumbbell graph). The argument proceeds as before, but we glue $X$ to $Y$ along $S^1 \coprod S^1$, and $G_i$ to $T$ along $e_1 \coprod e_2$. Thus we have a diagram
\begin{equation}
  \xymatrix{
    \cF_\infty^{top}(X\coprod_{S^1\coprod S^1} Y) \ar[r] \ar@/^1.5pc/[rr]^-\Psi \ar[d] & \cF_\infty^{top}(Y) \ar[d] \ar@/^1.5pc/[rr]^-\Psi & \cB(G_i\coprod_{E\coprod E} T) \ar[r] \ar[d] & \cB(T) \ar[d] \\
    \cF_\infty^{top}(X) \ar[r] \ar@/_1.5pc/[rr]_-\Psi &  \cF_\infty^{top}(S^1 \coprod S^1) \ar@/_1.5pc/[rr]_-\Psi & \cB(G_i) \ar[r]  & \cB(e_1 \coprod e_2)\\
  }
\end{equation}
where the two squares are  fiber products and the curved arrows are equivalences. 

It remains to analyze the restriction functors. If $e'$ is an infinite edge of $G_{i+1} = G_i \coprod_{E\coprod E} T$, then the restriction to $e'$ factors through restriction either to $G_i$ or $T$. Similarly, we claim that the restriction from $\cF_\infty^{top}(X \coprod_{S^1 \coprod S^1} Y)$ factors through restriction either to $X$ or $Y$. The only issue here is that we may not be able to choose a skeleton that has disjoint cycles at three punctures simultaneously. This occurs when we consider the third puncture of $\Sigma(T)$, since $Y$ is a dumbbell graph, or if $X$ has only three punctures. On the other hand, the modification we need to do in order to produce a puncture at $p(e')$ can be localized in a neighborhood of either $X$ or $Y$ inside $X \coprod_{S^1 \coprod S^1} Y$. Since an open restriction followed by a closed restriction is a closed restriction (see Proposition \ref{prop:compr}), it suffices to understand the closed restriction functor from a neighborhood of $X$ or $Y$ to the puncture. After restricting to a small enough neigborhood of $X$ or $Y$, the closed restriction to $X$ or $Y$ consists then of merely removing some noncompact edges of the skeleton, and it makes no difference whether we do this before or after modifying the skeleton. Thus the restriction the puncture $p(e')$ factors through restriction first to $X$ or $Y$. The rest of the argument is the same as in the previous case.
\end{proof}

We are now ready to prove our main theorem. We use the notations of Section \ref{sec tcy}. 
Let $(X_\cT, W_\cT)$ be a toric Calabi-Yau LG model, and let $\Sigma_\cT$ be the mirror curve.

\begin{theorem}
\label{mainmainmain}
There is an equivalence of categories
$$
MF(X_\cT, W_\cT) \simeq Fuk^{top}(\Sigma_\cT). 
$$
\end{theorem} 
\begin{proof}
Let $G_\cT$ be the tropical curve dual to the triangulation $\cT$. Recall that $Fuk^{top}_\infty(\Sigma_\cT)$ denotes the Ind-completion of $Fuk^{top}(\Sigma_\cT)$. By Theorem \ref{gluing MF} and Theorem \ref{the induction} there are equivalences 
$$
MF^\infty(X_\cT, W_\cT) \simeq \cB(G_\cT) \simeq Fuk^{top}_\infty(\Sigma_\cT). 
$$
They restrict to an equivalence between the categories of  compact objects inside $MF^\infty(X, W)$ and $Fuk_{top}^\infty(\Sigma_\cT)$ 
$$
MF(X_\cT, W_\cT) \simeq  Fuk^{top}(\Sigma_\cT), 
$$
and this concludes the proof. 
\end{proof}

\begin{remark}
Let $\Sigma_\cT$ be an unramified  
cyclic cover of a punctured surface and let $Fuk^{wr}(\Sigma_\cT)$ be the wrapped Fukaya category. By \cite{AAEKO}  there is an equivalence 
$$
Fuk^{wr}(\Sigma_\cT) \simeq MF(X_\cT, f_\cT). 
$$ 
Together with Theorem \ref{mainmainmain}, this yields an equivalence 
$$
Fuk^{wr}(\Sigma_\cT) \simeq Fuk^{top}(\Sigma_\cT). 
$$
This establishes Kontsevich's claim \cite{Ko}, according to which the topological Fukaya category is equivalent the wrapped Fukaya category, for a large class of punctured Riemann surfaces. 
In her thesis Lee \cite{Le} extends the results of \cite{AAEKO} to all genera. This, combined with  Theorem \ref{mainmainmain}, gives a complete proof of  Kontsevich's claim for punctured surfaces. 
A different proof of this result, with different methods, was given in \cite{HKK}. 
\end{remark}

\begin{remark}
\label{lastremark}
Let us return to the picture of HMS for for partially wrapped Fukaya categories delineated in Remark \ref{firstlastremark}. As we explained there, the semi-infinite edges of a non-compact 
skeleton $\widetilde{S}$ of $\Sigma_\cT$  correspond to a smooth  stacky partial compactification $\widetilde{X}_\cT$ of $X_\cT$ with the following  property: each puncture $p$ of 
$\Sigma_\cT$ corresponds to an irreducible component 
$\mathbb{A}^1$ of the singular locus of $W_\cT$.\footnote{Such a compactification is not unique, but this does not affect the present discussion: any choice of these compactifications will give rise to the same derived category of singularities.} If $S$ has $n$ non-compact edges approaching the puncture $p$, then we compactify that copy of $\mathbb{A}^1$  
to a stacky rational curve $\mathbb{P}^1(1,n)$. That is, we cap the $\mathbb{A}^1$  with a copy of $[\mathbb{A}^1/\mu_n]$.

Let $(X_\cT)_0$ be the zero-fiber of $X_{\cT}$, and denote 
by $(\widetilde{X}_\cT)_0$ its compactification inside $\widetilde{X}_\cT$. As we explained in Remark \ref{firstlastremark}, we expect that the general statement of HMS   for the partially wrapped category $\cF(\widetilde{S})$ can be formulated as an equivalence 
\begin{equation}
\label{partwraphmstt2}
\cF(\widetilde{S}) \simeq D^b_{sg}((\widetilde{X}_\cT)_0). 
\end{equation}
Equivalence (\ref{partwraphmstt2}) should follow from our Theorem \ref{mainmainmain} and exotic gluing (Theorem \ref{gluing-main}). Assume for simplicity that the non-compact edges of $\widetilde{S}$ approach  a unique puncture $p$ of $\Sigma_\cT$: the general case is proved via an iteration of the argument. Removing the non-compact edges from $\widetilde{S}$ yields a compact skeleton $S$ of $\Sigma_\cT$. By modifying $\widetilde{S}$, and hence $S$, in the interior of $\Sigma_{\cT}$, it is possible to arrange that $S$ has a cycle at $p$, and that $\widetilde{S}$ consists of $S$ with several non-compact edges approaching $p$. The derived category of singularities should enjoy the same descent properties as the category of matrix factorizations: although this result as such does not seem to appear in the literature, it should follow from the techniques of \cite[Appendix A]{P}. Then Zariski descent for $D^b_{sg}(-)$ and Theorem \ref{gluing-main}   give, respectively,  the following two equivalences 
$$
D^b_{sg}((\widetilde{X}_\cT)_0) \simeq 
D^b_{sg}( (X_\cT)_0) \times_{\Perf(\mathbb{G}_m)} \Perf([\mathbb{A}^1/\mu_n])
$$
$$
\cF(\widetilde{S}) \simeq  Fuk^{top}(\Sigma_\cT) \times_{\cF(\Lambda(0,0))} \cF(\Lambda(0,n))
$$
By Theorem \ref{mainmainmain} we have that $D^b_{sg}( (X_\cT)_0) \simeq MF( X_\cT,  W_\cT) \simeq Fuk^{top}(\Sigma_\cT) $. Additionally there is an equivalence $\Perf([\mathbb{A}^1/\mu_n]) \simeq \cF(\Lambda(0,n))$. Thus we can deduce 
equivalence (\ref{partwraphmstt2}) exactly as in the proof of Theorem \ref{mainmainmain}.
\end{remark}


\begin{thebibliography}{AAEKO}

\bibitem[AAEKO]{AAEKO} 
M. Abouzaid, D. Auroux, A. I. Efimov, L. Katzarkov, D. Orlov 
``Homological mirror symmetry for punctured spheres,'' J. Amer. Math. Soc. 26, 1051--1083, (2013).

\bibitem[BFN]{BFN} D. Ben-Zvi, J. Francis, and D, Nadler. ``Integral transforms and Drinfeld centers in derived algebraic geometry," J. Amer. Math. Soc. 23, 909--966, (2010).

\bibitem[BJMS]{BJMS}
S. Brodsky, M. Joswig, 
R. Morrison, and B.Sturmfels, 
``Moduli of tropical plane curves," Res. in 
Math. Sci., Vol. 2 (1), 
1--31, (2015). 

\bibitem[B]{Bo} R. Bocklandt, with an appendix by Mohammed Abouzaid, 
``Noncommutative mirror symmetry for punctured surfaces,''
 Trans. Amer. Math. Soc. 368, 429--469,  (2016). 
 
\bibitem[BS]{BS} V. Bouchard, P. Su\l{}kowski, ``Topological recursion and mirror curves.''
Adv. Theor. Math. Phys.
Vol. 16 (5), 1443--1483, (2012).

\bibitem[C]{C} A. Connes,  ``Noncommutative Geometry,''  Academic Press, San Diego, New
York, London, 1994.

\bibitem[Co]{Co} L. Cohn, ``Differential graded categories are k-linear stable infinity categories,''  {\tt arXiv:1308.2587} 
  
 \bibitem[D]{D} T. Dyckerhoff, 
 ``$\bA^1$-homotopy invariants of topological Fukaya categories of surfaces,''  Comp. Math. 
Vol. 153 (8), 
 1673--1705, (2017). 



\bibitem[DK]{DK} T. Dyckerhoff, M. Kapranov, 
``Triangulated surfaces in triangulated categories,'' Journal of the European Mathematical Society, 
Vol. 20 (6), 1473--1524, (2018).



\bibitem[E]{E} 
D. Eisenbud, ``Homological algebra on a complete intersection, with an application to group representations,''  Trans. Amer. Math. Soc., Vol. 260 (1), 35--64, (1980).

\bibitem[DG]{DG} 
V. Drinfeld, and D. Gaitsgory ``Compact generation of the category of $\mathrm {D} $-modules on the stack of $ G $-bundles on a curve,'' Cambridge Journal of Mathematics, Vol.3 (1), 19-125, (2015). 

\bibitem[GR]{GR} 
D. Gaitsgory,  N. Rozenblyum, 
``A study in derived algebraic geometry: Volume I: correspondences and duality.'' No. 221. American Mathematical Soc., (2017).

\bibitem[Ge]{Ge} F. Genovese, ``Adjunctions of quasi-functors between dg-categories,''  Applied Categorical Structures, 
  Vol. 25 (4), 625--657,  (2017). 
  
  
\bibitem[GKR]{GKR} M. Gross, L. Katzarkov, H. Ruddat, ``Towards mirror symmetry for varieties of general type,'' Advances in Mathematics, 308, 208--275,  (2017).

\bibitem[HKK]{HKK} 
F. Haiden, L. Katzarkov, M. Kontsevich, ``Flat surfaces and stability structures,'' Publications math\' ematiques de l'IH\'ES, Vol. 126 (1), 247--318,  (2017).

\bibitem[H]{H} R. Haugseng,  ``Rectification of enriched $\infty$-categories,'' 
  Algebraic \& Geometric Topology, 
Vol. 15 (4),
 1931--1982, (2015).
 
\bibitem[HV]{HV} 
K. Hori, C. Vafa, ``Mirror Symmetry,'' 
{\tt arXiv:hep-th/0002222}


\bibitem[K]{Ko} M. Kontsevich, ``Symplectic Geometry of Homological Algebra,''
lecture at Mathematische Arbeitstagung 2009; notes at \\
http://www.mpim-bonn.mpg.de/Events/This+Year+and+Prospect/AT+2009/AT+program/.

\bibitem[Ku]{Ku}
T. Kuwagaki, ``The nonequivariant coherent-constructible correspondence for toric stacks,'' 
 {\tt arXiv:1610.03214}. 
 
 \bibitem[Le]{Le} 
H. Lee, ``Homological mirror symmetry for Riemann surfaces from pair of pants decompositions,'' doctoral dissertation, University of California at Berkeley, available at 
{\tt http://escholarship.org/uc/item/74b3j149\#page-1} and {\tt arXiv:1608.04473}.

\bibitem[LeP]{LeP}
Y. Lekili, A. Polishchuck, 
``Auslander orders over nodal stacky curves and partially wrapped Fukaya categories,'' 
 Journal of Topology, Vol. 11 (3), 
 615--644, (2018). 

\bibitem[LP]{LP} 
K. Lin, D. Pomerleano, 
``Global matrix factorizations,''
Math. Res. Lett., Vol. 20 (1), 91--106, (2013). 

\bibitem[LS]{LS}  
V. Lunts, O. Schn\"urer,  ``Matrix factorizations and semi-orthogonal decompositions for blowing-ups,'' 
Journal of Noncommutative Geometry, Vol. 10 (3), 907--980, (2016).


\bibitem[Lu]{Lu} 
J. Lurie, ``Higher topos theory,'' Annals of Mathematics Studies, 170. Princeton University
Press 2009.



\bibitem[MP]{MP} M. Mulase, M. Penkava, 
``Ribbon graphs, quadratic differentials on Riemann surfaces, and algebraic curves defined over $\bar \bQ$,''
The Asian Journal of Mathematics vol. 2 (4), 875--920 (1998).

\bibitem[N]{N0} D. Nadler, ``Microlocal branes are constructible sheaves,'' Selecta Mathematica, Vol. 15 (4), 563--619, (2009).

\bibitem[N1]{N} D. Nadler, ``Cyclic symmetry of $A_n$ quiver representations,'' 
Advances in Mathematics 269, 346--363, (2015).


\bibitem[N2]{N5} D. Nadler, ``Arboreal singularities,'' Geometry \& Topology Vol. 21 (2), 1231--1274, (2017).


\bibitem[N3]{N2} D. Nadler, 
``A combinatorial calculation of the Landau-Ginzburg model $M$=$\bC^3$, W=$z_1 z_2 z_3$,'' Selecta Mathematica Vol. 23 (1), 519--532, (2017).

\bibitem[N4]{N3} D. Nadler, ``Mirror symmetry for the Landau-Ginzburg A-model $M$=$\mathbb{C}^n$, W=$z_1 ... z_n$,'' 
{\tt arXiv:1601.02977}.

\bibitem[N5]{N4} D. Nadler, 
``Wrapped microlocal sheaves on pairs of pants,'' {\tt arXiv:1604.00114}. 

\bibitem[NZ]{NZ} D. Nadler,  
E. Zaslow, ``Constructible Sheaves and the Fukaya Category,'' J. Amer. Math. Soc. 22,  233--286, (2009).

\bibitem[O1]{O1} D. Orlov, ``Triangulated categories of singularities and D-branes in Landau-Ginzburg models,'' Proc. Steklov Inst. Math., Vol. 246 (3), 227--248, (2004). 

\bibitem[O2]{O2} D. Orlov,
``Matrix factorizations for non affine LG models,'' Math. Ann. 353, 1, 95--108, (2012). 

\bibitem[P]{P} A. Preygel,
``Thom-Sebastiani \& Duality for Matrix Factorizations,'' {\tt arXiv:1101.5834}.

\bibitem[RSTZ]{RSTZ} H. Ruddat, N. Sibilla, D. Treumann, E. Zaslow, ``Skeleta of affine hypersurfaces,'' 
Geom. and Top.,  Vol. 18 (3), 1343--1395, (2014).  

\bibitem[Sh]{Sh} 
N. Sheridan, ``On the homological mirror symmetry conjecture for pairs of pants,'' J.   Diff. Geom., Vol. 89 (2), 271--367, (2011).

\bibitem[STZ]{STZ} N. Sibilla, D. Treumann, E. Zaslow, ``Ribbon graphs and mirror symmetry,'' 
Selecta Mathematica, Vol. 20 (4), 979--1002, (2014).

\bibitem[Sy]{Sy} Z. Sylvan, ``On partially wrapped Fukaya categories,'' {\tt arXiv:1604.02540}.

\bibitem[T]{T} G. Tabuada, ``Une structure de cat{\'e}gorie de modeles de Quillen sur la cat{\'e}gorie des dg-cat{\'e}gories,'' Comptes Rendus Mathematique,
  vol. 340 (1),
 15--19, (2005).
 
\bibitem[Ta]{Ta} D. Tamarkin, ``Microlocal category,'' 
{\tt arXiv:1511.08961}. 

\bibitem[To1]{To} B. To\"en, ``The homotopy theory of dg-categories and derived Morita theory,''
Invent. math., Volume 167, Issue 3, 615--667 (2007).

  
\bibitem[Ts]{Ts} B. Tsygan, ``A microlocal category associated to a symplectic manifold,'' 
{\tt arXiv:1512.02747}. 

\end{thebibliography}
\end{document}